\documentclass{amsart}

\usepackage{amsmath,amsfonts,graphicx}
\usepackage{commath}
\usepackage{mathtools}

\usepackage{xcolor}

\usepackage{changes}
\newcommand{\stkout}[1]{\ifmmode\text{\sout{\ensuremath{#1}}}\else\sout{#1}\fi}
\setdeletedmarkup{\stkout{#1}}

\usepackage{comment}
\usepackage{amssymb}

\usepackage{todonotes}
\usepackage{hyperref}
\usepackage[capitalise]{cleveref}

\usepackage{tikz}
\usetikzlibrary{decorations.pathreplacing,arrows.meta}
\tikzset{>={Stealth[scale=.4]}}
\usetikzlibrary{calc}

\theoremstyle{plain}
\newtheorem{theorem}{Theorem}[section] % reset theorem numbering for each chapter

\theoremstyle{definition}
\newtheorem{definition}[theorem]{Definition} % definition numbers are dependent on theorem numbers
\newtheorem{corollary}[theorem]{Corollary}
\newtheorem{proposition}[theorem]{Proposition}
\newtheorem{lemma}[theorem]{Lemma}

\theoremstyle{plain}
\newtheorem{maintheorem}{Theorem}

\crefname{maintheorem}{Theorem}{Theorems}
\Crefname{maintheorem}{Theorem}{Theorems}

\theoremstyle{remark}
\newtheorem*{remark}{Remark}

\newcommand{\R}{\mathbb{R}}

\newcommand{\N}{\mathbb{N}}

\newcommand{\dist}{\mathrm{dist}}

\begin{document}
\title[Thermodynamic formalism for endomorphisms with critical set]{Thermodynamic formalism for
Partially Hyperbolic Endomorphisms with Critical Set}

\author{Kecheng Li}
\address{Department of Mathematics \& Computer Science, Hendrix College, 
Conway, AR 72032}
\email{lik@hendrix.edu}
\urladdr{} % Delete if not wanted.

\begin{abstract}
    We study the thermodynamic formalism of a broad class of partially hyperbolic, non-uniformly expanding endomorphisms with critical set, and show that they admit unique equilibrium states for all H\"older potentials of small variation. In particular, our results yield uniqueness of equilibrium states for Viana maps with small variation H\"older potentials, for non-uniformly expanding interval maps with finitely many critical points and hyperbolic H\"older potentials, and for Benedicks--Carleson quadratic maps with arbitrary H\"older potentials.
\end{abstract}

\maketitle

\section{Introduction}\label{sec:main}
We study a class of partially hyperbolic, non-uniformly expanding endomorphisms with non-flat critical set. These systems admit an invariant dominated splitting with a one-dimensional center direction, and the main difficulty is that expansion along this center direction is only non-uniform: it may be interrupted by returns near the critical set, where the map folds and the derivative becomes small. Thus, although such systems are far from uniformly hyperbolic, they still retain enough hyperbolic structure to support a meaningful thermodynamic formalism. 

We establish existence and uniqueness of equilibrium states for potentials with sufficiently small variation for this class of systems. In particular, this includes the zero potential and hence yields uniqueness of the measure of maximal entropy. We also develop a one-dimensional counterpart for non-uniformly expanding interval maps with finitely many non-flat critical points. In this setting, the small variation assumption can be replaced by hyperbolic potential. This gives a stronger conclusion for the Benedicks--Carleson maps: every H\"older continuous potential admits a unique equilibrium state.

We now state the main results precisely:
\begin{maintheorem}\label{thm:main}
Let \(f:M\to M\) be a \(C^2\) endomorphism with critical set \(\mathcal{C}\), and let \(\Lambda\subset M\) be a compact forward-invariant set such that \(f|_\Lambda\) is \textbf{partially hyperbolic} with
\[
T_\Lambda M = E^u \oplus E^c, \qquad \dim E^c = 1,
\]
Assume moreover that:

\begin{enumerate}
    \item \(f\) is a local diffeomorphism on \(M\setminus \mathcal C\), and \(\mathcal C\) is \textbf{non-flat} (\cref{nonflat}) in the center direction;

    \item \textbf{non-uniformly expanding}: there exist a constant \(c_0>0\) and a subset \(H\subset \Lambda\) of positive Lebesgue measure such that for every \(x\in H\),
    \[
    \limsup_{n\to\infty}\frac1n\sum_{j=0}^{n-1}
    \log \|Df(f^j(x))^{-1}\|
    \le -c_0<0;
    \]

    \item \textbf{slow recurrence to the critical set}: for every \(\gamma>0\), there exists \(\delta>0\) such that for every \(x\in H\),
    \[
    \limsup_{n\to\infty}\frac1n\sum_{j=0}^{n-1}
    \bigl(-\log \operatorname{dist}_\delta(f^j(x),\mathcal C)\bigr)
    \le \gamma;
    \]

    \item \(f|_\Lambda\) is \textbf{topologically exact}: for every nonempty open set \(O\subset \Lambda\), there exists \(n=n(O)\in\mathbb{N}\) such that
    \[
    f^n(O)=\Lambda;
    \]

    \item \textbf{bounded unstable Jacobian distortion}: there exists \(\kappa\in(0,c_0/2)\) such that for every \(x,y\in \Lambda\),
\[
e^{-\kappa}
<
\frac{\left|\det\bigl(Df_x|_{E^u(x)}\bigr)\right|}
     {\left|\det\bigl(Df_y|_{E^u(y)}\bigr)\right|}
<
e^{\kappa}.
\]
\end{enumerate}

Then every H\"older continuous potential \(\varphi:\Lambda\to\mathbb R\) satisfying
\[
\sup \varphi-\inf \varphi < \frac{c_0}{2}-\kappa
\]
has a unique equilibrium state \(\mu_{\varphi}\).
\end{maintheorem}

When the unstable bundle is trivial, we also obtain a one-dimensional counterpart of the main result for non-uniformly expanding interval maps with finitely many critical points:

\begin{maintheorem}\label{thm:main-1d}
Let \(f:I\to I\) be a \(C^2\) interval map with finite non-flat critical set \(\mathcal C\), and let \(J\subset I\) be a compact forward-invariant interval such that \(f|_J\) is topologically exact. Assume that there exist a constant \(c_0>0\) and a subset \(H\subset J\) of positive Lebesgue measure such that:

\begin{enumerate}
    \item for every \(x\in H\),
    \[
    \limsup_{n\to\infty}\frac1n\sum_{j=0}^{n-1}\log |f'(f^jx)|^{-1}\le -c_0<0;
    \]

    \item for every \(\gamma>0\), there exists \(\delta>0\) such that for every \(x\in H\),
    \[
    \limsup_{n\to\infty}\frac1n\sum_{j=0}^{n-1}
    \bigl(-\log \operatorname{dist}_\delta(f^j(x),\mathcal C)\bigr)\le \gamma.
    \]
\end{enumerate}

Then every H\"older continuous potential \(\varphi:J\to\mathbb R\) satisfying
\[
\max_{\nu\in\mathcal M(f|_J)}\int\varphi\,d\nu<P_{\mathrm{top}}(\varphi;f|_J)
\]
has a unique equilibrium state \(\mu_\varphi\) with positive entropy. Moreover, \(\mu_\varphi\) satisfies the upper level-2 large deviations bound for weak*-closed convex sets stated in \cref{ct}.
\end{maintheorem}

\begin{remark}[Hyperbolic potentials]
A continuous potential \(\varphi:X\to\mathbb R\) is called \emph{hyperbolic} for \(f\) if
\[
\max_{\nu\in\mathcal M(f)}\int\varphi\,d\nu<P_{\mathrm{top}}(\varphi;f).
\]
Equivalently, \(\sup_X \frac1N S_N\varphi<P_{\mathrm{top}}(\varphi;f)\) for some \(N\ge1\); see \cite{LRL14a}.
\end{remark}

We establish uniqueness of equilibrium states through the orbit decomposition and pressure gap framework developed in \cite{CT16}, which reduces the problem to constructing a suitable orbit decomposition and verifying the required specification and pressure conditions. To carry out this program, we draw on the theory of hyperbolic times from \cite{ABV00}. This theory provides the non-uniform expansion estimates needed to detect invariant measures with high pressure and to identify good orbit segments with sufficient hyperbolic behavior.

One benefit of our approach is the construction of the unique equilibrium state as a weak Gibbs measure. In \cite{Bow74}, the lower Gibbs property is essential in ruling out mutually singular equilibrium states. This approach is generalized in \cite{CT16}, in which the authors derive the lower Gibbs property of equilibrium state for the ``essential collection of orbit segments'' which dominates in pressure as well as a global upper Gibbs property for all orbit segments.

We now review earlier results related to \cref{thm:main-1d}. In the constant potential case, the conclusion of \cref{thm:main-1d} is classical: for topologically exact multimodal interval maps, uniqueness of the measure of maximal entropy is due to Hofbauer \cite{Hof79, Hof81}. For general potentials, Li and Rivera-Letelier established existence and uniqueness of equilibrium states for hyperbolic potentials \cite{LRL14a}, and showed that under a weak hyperbolicity assumption on the map every Hölder continuous potential is hyperbolic, whence it admits a unique equilibrium state with exponential decay of correlations and no phase transition \cite{LRL14b}.

\cref{thm:main-1d} recovers the above results. Its contribution lies less in the class of admissible potentials than in the method: uniqueness is obtained by the orbit decomposition approach, under non-uniform expansion and slow recurrence to the critical set. Moreover, \cref{thm:main-1d} is the one-dimensional counterpart of \cref{thm:main} with a weaker restriction on the potential. The construction also yields Gibbs type estimates on sufficiently long orbit segments in the good collection $\mathcal G$.

\subsection{Applications}
In addition to the Gibbs construction, another advantage of the approach in this paper is that it applies in a general geometric setting and does not rely on symbolic models or inducing schemes tailored to a specific example. Once the abstract framework is established, it can be applied to concrete families of non-uniformly expanding maps. In particular, we show that the hypotheses of the main theorem are satisfied by Viana maps and the Benedicks--Carleson quadratic family. As a consequence, Viana maps admit unique equilibrium states for H\"older potentials of sufficiently small variation, while the Benedicks--Carleson maps considered here admit unique equilibrium states for every H\"older potential. In particular, both families have a unique measure of maximal entropy.

\subsubsection{Viana maps}
\begin{maintheorem}[\cite{Li26}]\label{Li26}
For every H\"older potential \(\varphi\) with sufficiently small variation, the Viana map \(f_{\alpha}\) and every map \(f\) in a sufficiently small \(C^{3}\)-neighborhood of \(f_{\alpha}\) admit a unique equilibrium state \(\mu_{\varphi}\) for \(\varphi\). Moreover, \(\mu_{\varphi}\) satisfies the upper level-2 large deviations bound of \cref{ct}.
\end{maintheorem}

In \cref{sec:viana-maps}, we review the main properties of Viana maps and verify that they satisfy the hypotheses of our main theorem.

A proof in the special case of Viana maps was previously obtained in \cite{Li26} by the author, and the present paper extends that argument to a broader higher-dimensional partially hyperbolic with critical set setting.

\subsubsection{Benedicks--Carleson quadratic family}
\begin{maintheorem}\label{BC}
Let \(f_a(x)=1-a x^2\), \(x\in[-1,1]\),
and let
\[
I_a=[f_a^2(0),f_a(0)]
\]
be the \emph{core interval}. For any Benedicks--Carleson parameter \(a\), every H\"older continuous potential \(\varphi:I_a\to\mathbb R\) has a unique equilibrium state \(\mu_\varphi\) for \(f_a|_{I_a}\). Moreover, \(\mu_\varphi\) has positive entropy and satisfies the upper level-2 large deviations bound for weak*-closed convex sets stated in \cref{ct}.
\end{maintheorem}

\section*{Acknowledgements}

The author thanks Juan Rivera-Letelier for his helpful comments on the doctoral thesis on which this paper is based. His suggestions concerning hyperbolic potentials led to the relaxation of the small variation assumption in the one-dimensional results.

\section{Thermodynamic formalism: topological pressure and equilibrium states}
Originating in Gibbs's description of physical systems such as the molecules in a tank of gas, thermodynamic formalism exploits an analogy between classical thermodynamics and dynamical systems to single out a range of dynamically interesting measures and quantities. The central quantity is the topological pressure, and the measures it distinguishes are the equilibrium states. Over the past decades thermodynamic formalism has become a mathematical framework in its own right, one that suggests the right theorems and, to some extent, their proofs, particularly for systems exhibiting hyperbolicity, uniform or beyond.

\subsection{Definition of topological pressure}\label{pressure}
Motivated by statistical mechanics, pressure is a key quantity in the formalism that acts like the specific Gibbs free energy, measuring the exponential growth rate of orbit segments under some given weight function.

To define the topological pressure, we first introduce necessary notations and definitions. Let $X$ be a compact metric space. A set \(E \subset X\) is called \((n, \epsilon)\)-separated if \(d_n(x, y) \;\ge\;\epsilon 
\) for all \(x, y \in E\) with \(x \neq y\). Given \(\mathcal{D} \subset X \times \mathbb{N}\), we interpret \(\mathcal{D}\) as a collection of finite orbit segments. Define \(\mathcal{D}_n \;=\; \bigl\{\, x \in X : (x, n) \in \mathcal{D} \bigr\}\) the set of initial points of all orbit segments of length \(n\) in \(\mathcal{D}\). Given $\varphi: X \rightarrow \R$, called the potential function, we then consider the corresponding partition sum:
\[
\Lambda_n^{\mathrm{sep}} \bigl(\mathcal{D}, \varphi, \epsilon; f\bigr)
\;=\;
\sup \Bigl\{ \sum_{x \in E} e^{S_n \varphi(x)} 
\;:\;
E \subset \mathcal{D}_n \text{ is } (n,\epsilon)\text{-separated} 
\Bigr\}.
\]
Here $S_n \varphi := \sum_{i=0}^{n-1}\varphi(f^ix)$ is called the Birkhoff sum, where one evaluate the potential function $\varphi$ along the orbit segment $(x, f(x),\cdots,f^{n-1}(x))$ and sum up. Note that when $\varphi \equiv0$, \(\Lambda_n^{\mathrm{sep}} \bigl(\mathcal{D}, 0, \epsilon; f\bigr)\) yields the cardinality of the maximal $(n, \epsilon)$-separated set. 

\begin{definition}[Topological pressure]
    The \emph{topological pressure} of \(\varphi\) on \(\mathcal{D}\) at scale \(\epsilon\) is given by
    \[
    P\bigl(\mathcal{D}, \varphi, \epsilon; f\bigr)
    \;=\;
    \limsup_{n \to \infty}
    \frac{1}{n}
    \log \Lambda_n^{\mathrm{sep}}\bigl(\mathcal{D}, \varphi, \epsilon; f\bigr).
    \]
    Then the \emph{pressure of \(\varphi\) on \(\mathcal{D}\)} is defined as
    \[
    P(\mathcal{D}, \varphi; f)
    \;=\;
    \lim_{\epsilon \to 0}\,
    P\bigl(\mathcal{D}, \varphi, \epsilon; f\bigr).
    \]
\end{definition}

Again note that if $\varphi\equiv 0$, then we obtain the topological entropy from the above definition. (In fact, by the cohomologous relation, any constant function yields the topological entropy). In that case, the topological entropy can be understood as the exponential growth rate of the number of distinguishable orbits of length $n$, which gives a measure of the rate of mixing of the system $(X, f)$.

\subsection{Equilibrium states}
Equilibrium state are Special invariant probability measures that maximize a combination of (Kolmogorov--Sinai) entropy and energy like potential functions. Instead of defining the entropy from in point view of information theory and using measurable partitions, we give an alternative definition of the entropy of measure by Katok. 

\subsubsection{Katok's Definition of Measure-Theoretic Entropy}
For $n \in \mathbb{N}$ and $x \in X$, the \textbf{Bowen metric} (or dynamical metric) $d_n$ is defined by
\[
d_n(x, y) = \max_{0 \le i < n} d(f^i(x), f^i(y)).
\]
The open dynamical ball of radius $\varepsilon > 0$ centered at $x$ is denoted as 
\(
B_n(x, \varepsilon) = \{y \in X : d_n(x, y) < \varepsilon\}
\). 

Given an ergodic invariant probability measure $\mu$ on $X$. For any $\varepsilon > 0$ and $0 < \delta < 1$, let $N_\mu(n, \varepsilon, \delta)$ denote the minimal number of $d_n$-balls of radius $\varepsilon$ needed to cover a set of $\mu$-measure at least $1 - \delta$:
\[
N_\mu(n,\varepsilon,\delta)
=
\min_{E\subset X \text{ finite}}\left\{
\#E : \;
\mu\left(\bigcup_{x\in E}B_n(x,\varepsilon)\right)
\ge 1-\delta
\right\}.
\]

The \textbf{measure-theoretic entropy} $h_\mu(f)$ can be defined as the exponential growth rate of this minimal number of balls:
\[
h_\mu(f) = \lim_{\varepsilon \to 0} \limsup_{n \to \infty} \frac{\log N_\mu(n, \varepsilon, \delta)}{n} = \lim_{\varepsilon \to 0} \liminf_{n \to \infty} \frac{\log N_\mu(n, \varepsilon, \delta)}{n}
\]

By Katok's theorem \cite[Theorem 1.1]{Kat80}, this limit exists, is independent of $\delta \in (0, 1)$, and coincides exactly with the classical measure theoretic entropy.

Katok's entropy formula yields a topological and geometric characterization of measure-theoretic entropy. The entropy $h_\mu(f)$ measures the exponential growth rate of the number of orbit segments (dynamical balls) needed to track a significant majority (\(1-\delta\)) of the system's trajectories. Comparing this formula with the covering definition of topological
entropy gives \(h_\mu(f)\le h_{\mathrm{top}}(f)\), which is one direction of the variational principle.

\subsubsection{Variational Principle and Equilibrium States}
    Let \(\mathcal{M}(f)\) be the set of \(f\)-invariant Borel probability measures, and \(\mathcal{M}_e(f)\) the subset of ergodic measures in \(\mathcal{M}(f)\). The variational principle \cite[Theorem 9.10]{Wal82} says that,
    \begin{equation}\label{variational principle}
    P_{\mathrm{top}}(\varphi; f)
    \;=\;
    \sup_{\mu \in \mathcal{M}(f)}
    \Bigl(
    h_\mu(f)
    \;+\;
    \int \varphi \,d\mu
    \Bigr)
    \;=\;
    \sup_{\mu \in \mathcal{M}_e(f)}
    \Bigl(
    h_\mu(f)
    \;+\;
    \int \varphi \,d\mu
    \Bigr).
    \end{equation}

\begin{definition}[Equilibrium State]
    A measure $\mu \in M(f)$ achieving the supremum in \eqref{variational principle} is called an \emph{equilibrium state}. If $\varphi = 0$, such a measure is called \emph{measure of maximal entropy} (MME).
\end{definition}

\section{Non-uniform Expansivity and Specification} \label{CT}

Bowen showed that if \((X, f)\) is expansive and has the specification property, and if \(\varphi\) satisfies certain regularity conditions, then there is a unique equilibrium state for $\varphi$ (\cite{Bow74}). We recall definitions and results from \cite{CT16}, which demonstrate that non-uniform analogues of Bowen's hypotheses are sufficient to guarantee uniqueness.

\subsection{Expansivity}   
Given a continuous map \(f : X \to X\), the \emph{infinite Bowen ball} of size \(\epsilon > 0\) around a point \(x \in X\) is defined by
\[
\Gamma_\epsilon(x)
\;:=\;
\bigl\{\, y \in X : d(f^k x, f^k y) < \epsilon 
\text{ for all } k \in \mathbb{\N} \bigr\}.
\]
If there exists \(\epsilon > 0\) such that \(\Gamma_\epsilon(x) = \{x\}\) for every \(x \in X\), then we say that \((X, f)\) is \emph{positively expansive}. Throughout the rest of the paper we drop ``positively'' and simply write \emph{expansive}, since all maps under consideration are non-invertible.

\begin{definition}
For \(f : X \to X\), the set of \emph{non-expansive points} at scale \(\epsilon\) is defined by
\[
\mathrm{NE}(\epsilon)
\;=\;
\bigl\{\, x \in X : \Gamma_\epsilon(x) \neq \{x\} \bigr\}.
\]
An \(f\)-invariant measure \(\mu\) is called \emph{almost expansive} at scale \(\epsilon\) if \(\mu\bigl(\mathrm{NE}(\epsilon)\bigr) = 0\).

Given a potential \(\varphi\), the \emph{pressure of obstructions to expansivity} at scale \(\epsilon\) is
\begin{align*}
    P_{\mathrm{exp}}^{\perp}(\varphi,\epsilon)
    \;&=\;
    \sup_{\mu \in \mathcal{M}_e(f)}
    \Bigl\{
    h_\mu(f) + \int \varphi \, d\mu
    : \mu\bigl(\mathrm{NE}(\epsilon)\bigr) > 0
    \Bigr\}
\end{align*}

This quantity is \emph{non-increasing} in \(\epsilon\), so we define a scale-free version by
\[
P_{\mathrm{exp}}^{\perp}(\varphi)
\;=\;
\lim_{\epsilon \to 0}
P_{\mathrm{exp}}^{\perp}(\varphi,\epsilon).
\]
\end{definition}

In particular, the scale-free version
\(
P_{\exp}^{\perp}(\varphi) 
\)
satisfies
\[
P_{\exp}^{\perp}(\varphi) \ge P_{\exp}^{\perp}(\varphi,\epsilon)
\]
for any fixed \(\epsilon > 0\).

Roughly speaking, \( P^\perp_{\exp}(f,\epsilon) \) can be understood as follows. Recall that positive expansivity is equivalent to the condition
\[
\bigcap_{n \ge 0} B_n(x,\epsilon) = \{x\},
\]
which holds for every \( x \in X \) when \(\epsilon\) is sufficiently small. We define \( P^\perp_{\exp}(f,\epsilon) \) to be the supremum of the pressures of all ergodic measures that assign positive weight to the set of points where this condition fails.

\subsection{Specification}
Specification describes the property that different Bowen balls can be connected by an orbit segment with uniform gap. 

\begin{definition}[W–specification at scale $\epsilon$]
A collection of orbit segments 
\[
   \mathcal D \subset X\times\N
   \quad\bigl(\text{with } (x,n)\text{ representing the segment }
   x,\,f(x),\dots,f^{n-1}(x)\bigr)
\]
has the \emph{$(W)$-specification property at scale $\epsilon>0$} if there exists 
\(\tau=\tau(\epsilon)\in\N\) such that, for every finite set of segments
\[
  \{(x_j,n_j)\}_{j=1}^{k}\subset\mathcal D,
\]
one can find ``gluing times'' 
\(\tau_1,\dots,\tau_{k-1}\in\N\) with \(\tau_i\le \tau\) for every
\(i\) and a point \(y\in X\) satisfying
\[
   d_{n_j}\!\bigl(f^{t_{j}}(y),x_j\bigr)<\epsilon
   \quad\text{for each }1\le j\le k,
\]
where
\[
  t_1:=0,\qquad 
  t_j:=\sum_{i=1}^{j-1}\bigl(n_i+\tau_i\bigr)\quad (2\le j\le k).
\]
Here \(d_{n}(x,y)=\max_{0\le m<n}d\bigl(f^{m}x,f^{m}y\bigr)\)
is the Bowen metric of length \(n\).

We say that \(\mathcal{D}\) has the \emph{$(W)$-specification property} if the above condition holds for every \(\epsilon > 0\). We say that \((X, f)\) has the $(W)$-specification property if \(X \times \mathbb{N}\) does. 
\end{definition}

We use \emph{specification} as shorthand for $(W)$-specification, allowing transition times to vary subject to a uniform upper bound $\tau$. In discrete time, this is weaker than $(S)$-specification, which requires every transition time to equal $\tau$. We do not require the shadowing orbit to be periodic.

\begin{definition}
A triple \(\bigl(\mathcal{P}, \mathcal{G}, \mathcal{S}\bigr) \subset (X \times \mathbb{N})^3\) is called a \emph{decomposition} for \((X, f)\) if for every \(x \in X\) and \(n \in \mathbb{N}\), there exist \(p, g, s \in \mathbb{N}\) with \(p + g + s = n\) such that
\[
(x, p) \;\in\; \mathcal{P},
\quad
\bigl(f^p x, g\bigr) \;\in\; \mathcal{G},
\quad
\bigl(f^{p+g} x, s\bigr) \;\in\; \mathcal{S}.
\]
In other words, every orbit segment can be decomposed into a “prefix” from \(\mathcal{P}\), a “good” core from \(\mathcal{G}\), and a “suffix” from \(\mathcal{S}\).

\end{definition}

Note that the symbol $(x, 0)$ denotes the empty set, and the functions $p, g, s$ are permitted to take the value zero.

\begin{definition}
The \emph{pressure of obstructions to specification} at scale \(\epsilon\) is defined by

\[
P_{\mathrm{spec}}^{\perp}(\varphi, \epsilon) 
\;:=\; 
\inf \Bigl\{ 
  P\bigl(\mathcal{P} \cup \mathcal{S},\varphi, 3\epsilon\bigr)
\Bigr\}.
\]

where the infimum is taken over $\mathcal{P}, \mathcal{S} \subset X \times \N$ such that there exists a decomposition $(\mathcal{P},\mathcal{G},\mathcal{S}$) for $(X,f)$ for which every $\mathcal{G}$ has specification at scale $\epsilon$. We then define the scale-free version:
\[
P_{\mathrm{spec}}^{\perp}(\varphi)
\;=\;
\lim_{\epsilon \to 0}
P_{\mathrm{spec}}^{\perp}(\varphi, \epsilon).
\]
\end{definition}
Again, since \(P_{\mathrm{spec}}^{\perp}(\varphi, \epsilon)\) is monotonic as a function of \(\epsilon\), this limit exists. In contrast to \(P_{\mathrm{exp}}^{\perp}(\varphi, \epsilon)\), \(P_{\mathrm{spec}}^{\perp}(\varphi, \epsilon)\) is non-decreasing in \(\epsilon\).

\subsection{Regularity for Potential}
\begin{definition}[Bowen Property]
Given \(\mathcal{D} \subset X \times \mathbb{N}\), we say a function \(\varphi : X \to \mathbb{R}\) has the \emph{Bowen property} on \(\mathcal{D}\) at scale \(\epsilon\) if there is a constant \(K = K(\varphi, \mathcal{D}, \epsilon)\) such that
\[
\bigl| S_n \varphi(x) - S_n \varphi(y) \bigr| < K
\]
for every \((x, n) \in \mathcal{D}\) and every \(y \in B_n(x, \epsilon)\). A function \(\varphi\) has the \emph{Bowen property} on \(\mathcal{D}\) if it satisfies the Bowen property on \(\mathcal{D}\) at some scale \(\epsilon\) (and thus at all smaller scales as well).
\end{definition}

\subsection{General Results on Uniqueness of Equilibrium States}
Our main tool for existence and uniqueness of equilibrium state is the orbit decomposition appraoch from \cite{CT16}. 

\begin{theorem}[{\cite[Theorem 5.5 and Theorem 5.7]{CT16}}]\label{ct}
Let \(X\) be a compact metric space and \(f \colon X \to X\) a continuous map. Let 
\(\varphi \colon X \to \mathbb{R}\) be a continuous potential function. Suppose that 
\[
P_{\mathrm{exp}}^{\perp}(\varphi) \;<\; P_{\mathrm{top}}(\varphi),
\]
and that \((X, f)\) admits a decomposition \(\bigl(\mathcal{P}, \mathcal{G}, \mathcal{S}\bigr)\) with the following properties:

\begin{enumerate}
    \item \(\mathcal{G}\) has $(W)$-specification at any scale;
    \item \(\varphi\) has the Bowen property on \(\mathcal{G}\);
    \item \(P(\mathcal{P} \cup \mathcal{S}, \varphi)\;<\; P_{\mathrm{top}}(\varphi; f )\).
\end{enumerate}

Then there is a unique equilibrium state \(\mu\) for \(\varphi\). Moreover, \(\mu\) satisfies the following upper large deviations bound:
if \(A\) is any weak*-closed and convex subset of the space of Borel probability
measures on \(X\), then
\[
\limsup_{n\rightarrow\infty} \frac{1}{n} \log\mu\left(\mathcal{E}_n^{-1}(A)\right) \le \sup_{\nu\in A\cap \mathcal{M}(f)}\left(h_{\nu}(f)+ \int \varphi \,d\nu - P_{\mathrm{top}}(\varphi; f )\right),
\]
where \(\mathcal{E}_n(x) = \frac{1}{n}\sum_{k=0}^{n-1} \delta_{f^kx}\) is the \(n\)-th empirical measure associated to \(x\).
\end{theorem}

We note that the theorem we apply assumes that the pressures of obstructions to both expansivity and specification are defined in a scale‐free manner.

\section{Non-uniformly Expanding Endomorphisms with Critical Set}
In this section, we introduce the notion of non-uniform expansion in the presence of a non-flat critical set. This notion was first formulated by Alves, Bonatti, and Viana in their seminal work \cite{ABV00}, where the main goal was to develop a framework for studying the existence of SRB measures under non-uniform expansion. One of their primary motivations came from Viana maps constructed earlier by Viana \cite{Via97}. The techniques introduced in \cite{ABV00}, however, have proved useful in much greater generality. Our presentation follows their framework in spirit, but the formulation is adapted to the partially hyperbolic setting considered in this paper.

\subsection{Non-flat Critical Set}
Let us explain first what we mean by critical set of a map. The critical set $\mathcal{C}$ may be thought of as containing those points $x$ where $Df(x)$ is non-invertible. 

Let $M$ be a compact manifold, $\mathcal{C} \subset M$ a compact subset with zero Lebesgue measure, and $f: M\setminus \mathcal{C} \rightarrow M$ a $C^2$ map on $M\setminus \mathcal{C}$. Recall in our partially hyperbolic setting we also assume that there is a continuous $Df$-invariant splitting $E^u \oplus E^c$ on $M$ and a continuous cone field $\mathrm{Cone}^c = \{\mathrm{Cone}^c(x)\}_{x\in M}$ around $E^c$ such that $\mathrm{Cone}^c(x)$ contains the critical (center) direction at points of $\mathcal{C}$.

\begin{definition} \label{nonflat}
    A critical set $\mathcal{C}$ of $f$ is said to be \emph{non-flat along the critical direction} if, in the critical direction, $f$ \emph{behaves like a power of the distance} to $\mathcal{C}$ close to the critical set $\mathcal{C}$, in the following sense: there exist constants $B >1$ and $\beta >0$ such that 
\begin{itemize}
        \item[(S1)] \quad $\frac{1}{B}\, \mathrm{dist}(x, \mathcal{C})^{\beta} \leq \frac{\|Df(x)v\|}{\|v\|} \leq B\, \mathrm{dist}(x, \mathcal{C})^{\beta};$
        \item[(S2)] \quad $| \log \|Df (x)^{-1}\| - \log \|Df (y)^{-1}\|| \leq B \frac{\mathrm{dist}(x,y)}{\mathrm{dist}(x, \mathcal{C})^{\beta}} ;$
        \item[(S3)] \quad $| \log |\det Df (x)^{-1}| - \log |\det Df (y)^{-1}|| \leq B \frac{\mathrm{dist}(x,y)}{\mathrm{dist}(x, \mathcal{C})^{\beta}} ;$
\end{itemize}
for every $v \in \mathrm{Cone}^c(x)\setminus\{0\}$ and $x,\, y \in M\setminus \mathcal{C}$ with $\dist(x,y)<\dist(x,\mathcal{C})/2$. 
\end{definition}

We refer to this kind of critical sets \emph{non-flat} since conditions (S1) to (S3) above are natural generalizations to arbitrary dimensions of the notion of non-flat critical point from one-dimensional dynamics.

Given $\delta > 0$ and $x \in M\setminus \mathcal{C}$ we define the \emph{$\delta$-truncated distance} from $x$ to $\mathcal{C}$
    \[\dist_{\delta}(x,\mathcal{C}) = 
    \begin{cases}
         1, &\quad \text{if \(dist(x,\mathcal{C}) \geq \delta\)},\\
         dist(x,\mathcal{C}), &\quad \text{otherwise}
    \end{cases};\]
Note that this is not really a distance function: $\dist(x,y)+\dist_{\delta}(y,\mathcal{C})$ may
be smaller than $\dist_{\delta}(x,\mathcal{C})$. We also denote $\mathcal{C}_{\infty}=\bigcup_{n=0}^{\infty}f^n(\mathcal{C})$.

\begin{definition}
A map $f$ is said to be \emph{non-uniformly expanding with slow recurrence to critical set $\mathcal{C}$} if $f$ satisfies (S1), (S2) and (S3), and the following conditions:
\begin{itemize}
    \item \textbf{Non-uniformly Expanding}: there exists $c > 0$
    \begin{equation}
    \limsup_{n\to \infty}\frac{1}{n}\sum^{n-1}_{j=0} \log \|D f(f^j(x))^{-1}\| \le -c
    \end{equation}
    for all $x$ in a positive Lebesgue measure set $H \subset M\setminus \mathcal{C}$,
    
    \item \textbf{Slow Recurrence to critical Set}: given any $\gamma>0$ there exists $\delta>0$ such that for every $x \in H$
    \begin{equation}
    \limsup_{n\rightarrow \infty} \frac{1}{n}\sum_{j=0}^{n-1}-\log(\dist_{\delta}(f^j(x),\mathcal{C})) \leq \gamma
    \end{equation}
\end{itemize} 
\end{definition}

\subsection{Hyperbolic Times}
In the presence of a critical set $\mathcal{C}$, we must ensure that iterates $f^j(x)$ with $0 \le j < n$ are not too close to $\mathcal{C}$.

Let $B >1$ and $\beta>0$ be as in the hypotheses (S1), (S2), (S3). In what follows $b$ is any fixed constant such that $0 <b <\min\{\frac{1}{2},\, \frac{1}{2\beta}\}$.

The following notion will allow us to derive uniform behavior (expansion, distortion) from the hypothesis of non-uniform expansion and slow recurrence to critical set. 
\begin{definition}\label{hyp times}
Given $0 < \sigma < 1$ and $\delta > 0$, we say that $n$ is a $(\sigma,\delta)$-hyperbolic time for $x\in M\setminus\mathcal{C}$ if, for all $1 \leq k \leq n$,
$$\prod_{j=n-k}^{n-1}\|Df(f^j(x))^{-1}\| \leq \sigma^k$$ 
and 
$$\dist_{\delta}(f^{n-k}(x),\mathcal{C}) \geq \sigma^{bk}$$
\end{definition}

Hyperbolic times are useful because they have the following key property:
\begin{lemma}{\cite[Lemma 5.2]{ABV00}} \label{hyp_times_prop}
     Given $\sigma < 1$ and $\delta > 0$, there exists $\delta_1 > 0$ such that if $n$ is a $(\sigma,\delta)$-hyperbolic time of $x\in M\setminus\mathcal{C}$, then there exists a neighborhood $V_x$ of $x$ such that:
    \begin{itemize}
        \item $f^n$ maps $V_x$ diffeomorphically onto the ball $B_{\delta_1}(f^n(x))$;
        \item For every $1\leq k < n$ and $y,z\in V_x$, $$\dist(f^{n-k}(y), f^{n-k}(z)) \leq \sigma^{k/2} \dist(f^n(y), f^n(z))$$.
    \end{itemize}
\end{lemma}

\begin{proof}
    We only outline the proof strategy here and refer readers to \cite[Lemma 5.2]{ABV00} for details. The idea is to prove that, by induction on $j \ge 1$, if $\delta_1$ is chosen small enough then there exists a well defined branch of $f^{-j}$ on the ball of radius $\delta_1$ around $f^n(x)$, mapping $f^n(x)$ to $f^{n-j} (x)$. In addition, this branch is a $\sigma^{j/2}$-contraction.
\end{proof}

A key point is that if $n$ is a hyperbolic time for a point $x$ then there exists a neighborhood $V_x$ of $x$ which is mapped onto a ball of fixed radius around $f^n(x)$ diffeomorphically and with uniformly bounded distortion.

In our setting, the non-flatness assumptions are verified only along the center direction, namely on a \(Df\)-invariant cone field \(\mathrm{Cone}^c\) around \(E^c\). However, the hyperbolic-time machinery in \cite{ABV00} is formulated using the full operator norm \(\|Df^{-1}\|\). Thus, in order to apply \cite[Lemma 5.2]{ABV00}, we must relate the full inverse norm to the inverse norm restricted to the center cone. The next lemma provides this comparison: since \(E^u\) is uniformly expanding and the splitting \(E^u\oplus E^c\) is dominated, the largest inverse expansion is realized, up to a uniform constant, along the center cone.

\begin{lemma}\label{cone-full-norm}
Assume that \(f|_\Lambda\) is partially hyperbolic with dominated splitting: \(
T_\Lambda M=E^u\oplus E^c\), \(\dim E^c=1,\) and let \(\mathrm{Cone}^c\) be a \(Df\)-invariant cone field around \(E^c\). Then there exists \(C\ge 1\) such that for every \(x\in \Lambda\setminus \mathcal C\),
\[
\sup_{v\in \mathrm{Cone}^c(x),\,\|v\|=1}\|Df(x)^{-1}v\|
\;\le\;
\|Df(x)^{-1}\|
\;\le\;
C\,\sup_{v\in \mathrm{Cone}^c(x),\,\|v\|=1}\|Df(x)^{-1}v\|.
\]
In particular, the notion of \((\sigma,\delta)\)-hyperbolic time defined using the full operator norm \(\|Df^{-1}\|\) is equivalent, up to adjusting \(\sigma\), to the one defined using the center cone field.
\end{lemma}

\begin{proof}
The first inequality is immediate, since \(\mathrm{Cone}^c(x)\) is a subset of the unit tangent space at \(x\):
\[
\sup_{v\in \mathrm{Cone}^c(x), \|v\|=1}\|Df(x)^{-1}v\|
\;\le\;
\|Df(x)^{-1}\|.
\]

For the converse inequality, write any unit vector \(w\in T_xM\) as
\[
w = w^u + w^c,
\qquad
w^u\in E^u(x),\; w^c\in E^c(x).
\]
Since the angle between \(E^u\) and \(E^c\) is uniformly bounded away from zero on \(\Lambda\), there exists \(K\ge1\) such that \(\|w^u\|+\|w^c\| \le K\) for every unit vector $w$. Hence
\[
\|Df(x)^{-1}w\|
\le
\|Df(x)^{-1}w^u\|+\|Df(x)^{-1}w^c\|
\le
K\Bigl(\|Df(x)^{-1}|_{E^u(x)}\|+\|Df(x)^{-1}|_{E^c(x)}\|\Bigr).
\]

By dominated splitting, there exists \(C_0\ge1\) such that \(\|Df(x)^{-1}|_{E^u(x)}\|\le C_0\,\|Df(x)^{-1}|_{E^c(x)}\|\) for every $x$. Therefore
\[
\|Df(x)^{-1}w\|
\le
K(C_0+1)\,\|Df(x)^{-1}|_{E^c(x)}\|.
\]
Since \(E^c(x)\subset \mathrm{Cone}^c(x)\), we have
\[
\|Df(x)^{-1}|_{E^c(x)}\|
\le
\sup_{v\in \mathrm{Cone}^c(x), \|v\|=1}\|Df(x)^{-1}v\|.
\]
Combining the previous estimates gives
\[
\|Df(x)^{-1}w\|
\le
K(C_0+1)\sup_{v\in \mathrm{Cone}^c(x), \|v\|=1}\|Df(x)^{-1}v\|.
\]
Taking the supremum over all unit vectors \(w\in T_xM\) yields
\[
\|Df(x)^{-1}\|
\le
C\,\sup_{v\in \mathrm{Cone}^c(x), \|v\|=1}\|Df(x)^{-1}v\|,
\]
where \(C=K(C_0+1)\). This proves the claim.
\end{proof}

In view of this equivalence, we henceforth use the term \emph{hyperbolic times} to refer to hyperbolic times relative to the center cone field.

\begin{definition}\label{expanding_set}
    For a fixed $\sigma<1$, we define the $\sigma$-expanding set $H(\sigma)$ to be the set of points $x\in M\setminus\mathcal{C}$ with the following two properties:
    \begin{equation}\label{non-unif. expand}
        \limsup_{n\rightarrow \infty} \frac{1}{n}\sum_{j=0}^{n-1}\log\|Df(f^j(x))^{-1}\| \leq 3\log\sigma < 0
    \end{equation}
    and, for any $\gamma > 0$ there is some $\delta > 0$ satisfying
    \begin{equation}\label{slow rec}
    \limsup_{n\rightarrow \infty} \frac{1}{n}\sum_{j=0}^{n-1}-\log(dist_{\delta}(f^j(x),\mathcal{C})) \leq \gamma
    \end{equation}
\end{definition}
The following lemma, due to Pliss \cite{Pli72}, is the key to prove that a point $x$ satisfying assumption \eqref{non-unif. expand} and \eqref{slow rec} has many (positive density at infinity) hyperbolic times.

\begin{lemma}[Pliss Lemma]\label{pliss}
    Given $A \ge c_2 >c_1 >0$, let $\theta_0 =(c_2-c_1)/(A-c_1)$. Then, given any real numbers $a_1,\cdots,a_N$ such that
    \[\sum_{j=1}^N a_j \ge c_2 N \quad \text{and}\quad a_j \le A \text{ for every } 1 \le j \le N,\]
there are $l >\theta_0 N$ and $1 <n_1 <\cdots<n_l \le N$ so that
\[\sum_{j=n+1}^{n_i}a_j \ge c_1(n_i-n) \quad\text{for every }0 \le n <n_i  \text{ and  }i=1,\cdots,l.\]
\end{lemma}

The next lemma asserts that points in $H(\sigma)$ have many $(\sigma,\delta)$-hyperbolic times.

\begin{lemma}\label{lem:infty_hyp_time} 
    Given $\sigma < 1$, there exist $\nu > 0$ and $\delta > 0$ depending only on $\sigma$ and $f$ such that, given any $x \in H(\sigma)$ and $N \geq 1$ sufficiently large, there are $(\sigma,\delta)$-hyperbolic times $1\leq n_1 <\cdots <n_l \leq N$ for $x$ with $l\geq \nu N$.
\end{lemma}

\begin{proof}
    Again we sketch the proof and refer readers to \cite[Lemma 5.4]{ABV00} for details. The strategy is to use \cref{pliss} twice, first for the sequence given by $a_j =-\log \|D f(f ^{j-1}(x))^{-1}\|$ (up to a cut off that makes it bounded from above), and then with $a_j =\log \dist_{\delta}(f^{j-1}(x),\mathcal{C})$ for a convenient $\delta > 0$. It can be proven that there exist many times $n_i$ for which the conclusion of Pliss’ holds, simultaneously, for both sequences. Then it can be checked that any such $n_i$ is a $(\sigma,\delta)$-hyperbolic time for $x$.
\end{proof}

Our argument relies on applying Pesin's entropy formula to suitable ergodic absolutely continuous invariant measures. The following result of Alves, Bonatti, and Viana \cite{ABV00} guarantees the existence of such measures.
\begin{theorem}\label{ABV ACIM}
    Assume $f$ is a non-uniformly expanding map with critical set $\mathcal{C}$, then Lebesgue almost every point belongs in the basin of some ergodic absolutely continuous invariant measure.
\end{theorem}

As a by-product, corresponding to the case when the critical set is empty,
\begin{corollary}
     Let $f : M \rightarrow M$ be a $C^2$ covering map (local diffeomorphism) on a compact manifold $M$, which is non-uniformly expanding: \eqref{non-unif. expand} holds for all $x$ in a set $H \subset M$ with positive Lebesgue measure. Then Lebesgue almost every point $x \in H$ belongs in the basin of some ergodic absolutely continuous invariant measure.
\end{corollary}

\section{Proof of the Main Theorems}

In this section, we prove \cref{thm:main,thm:main-1d} by verifying the expansivity, specification, pressure, and regularity conditions in the Climenhaga--Thompson criterion. We give the argument in the setting of \cref{thm:main}. For \cref{thm:main-1d}, we take the unstable direction to be trivial and replace the pressure comparison \cref{low_pressure_outside} by an estimate adapted for the interval maps. The resulting proof is completed in \cref{sec:verif}. 

In \cref{sec:exp}, we establish non-uniform expansivity by showing that any ergodic measure assigning positive weight to the non-expansive set has free energy strictly smaller than the topological pressure. In \cref{sec:spec}, we prove non-uniform versions of specification and regularity by decomposing the space of orbit segments \(X \times \mathbb{N}\) into a good core \(\mathcal{G}\) and a suffix collection \(\mathcal{S}\). The collection \(\mathcal{G}\) satisfies specification, and the potential \(\varphi\) has the Bowen property on \(\mathcal{G}\), while \(\mathcal{S}\) carries pressure strictly smaller than the topological pressure. Finally, in \cref{sec:verif}, we combine these ingredients to complete the proof.

\subsection{Non-uniform Expansivity}\label{sec:exp}
We first prove that, for maps in our class, any ergodic \(f\)-invariant measure with large pressure must be positively expansive. We follow the strategy outlined below:
\begin{itemize}
    \item Select a certain set \(\mathcal{K}\) of ergodic \(f\)-invariant measures, which are natural candidates for attaining the variational principle.
    \item Prove that any measure outside \(\mathcal{K}\) has strictly smaller free energy than the topological pressure and hence cannot be an equilibrium state.
    \item Show that any measure inside \(\mathcal{K}\) is an \emph{expanding measure}, and any expanding measure is almost positively expansive.
\end{itemize}

\begin{remark}
    The proof uses two main results from \cite{ABV00}. The first is the existence of some SRB measure \(\mu_{\mathrm{SRB}}\) (or, in our non-contracting setting, the corresponding smooth measure), whose free energy admits a good lower bound and is used to derive a pressure gap for any measure not in $\mathcal{K}$. The second is the existence of infinitely many, in fact dense, hyperbolic times for a positive Lebesgue measure set of points. This will allow us to show that any measure in \(\mathcal{K}\) assigns zero weight to the non-expansive set.
\end{remark}

We begin by describing \(\mathcal{K}\), the set of candidate measures for achieving the variational principle. Fix \(c_0>0\) and define
\[
\mathcal{K}
\;=\;
\Bigl\{
  \mu \in \mathcal{M}_e(f)
  :
  \lambda^c(\mu)
  \;>\;
  \tfrac{c_0}{3}
\Bigr\}.
\]
That is, \(\mathcal{K}\) consists of those ergodic \(f\)-invariant measures whose central Lyapunov exponent is strictly bounded below by \(\tfrac{c_0}{3}\).

The following lemma guarantees that $\mathcal{K}$ is not empty.
\begin{lemma}\label{existence_SRB}
There exists an ergodic SRB measure \(\mu_{\mathrm{SRB}}\). Moreover, since \(\dim E^c=1\) and \(E^u\) is uniformly expanding, we have
\[
\lambda^c(\mu_{\mathrm{SRB}})\ge c_0.
\]
In particular, \(\mathcal K\neq \emptyset\).
\end{lemma}

\begin{proof}
By \cref{ABV ACIM}, there exists an absolutely continuous invariant measure whose ergodic components are SRB measures, and whose basin covers Lebesgue almost every point of \(H\). Since \(H\) has positive Lebesgue measure, at least one ergodic SRB measure \(\mu_{\mathrm{SRB}}\) exists.

Because the non-uniform expansion is bounded away from \(c_0>0\) on \(H\), and \(E^u\) is uniformly expanding while \(\dim E^c=1\), the center Lyapunov exponent of \(\mu_{\mathrm{SRB}}\) is bounded below by \(c_0\). Therefore \(\mu_{\mathrm{SRB}}\in \mathcal K\), and so \(\mathcal K\neq\emptyset\).
\end{proof}

\subsubsection{\texorpdfstring{Low Pressure Outside \(\mathcal{K}\)}{Low Pressure Outside K}}
The measures in \(\mathcal{K}\) are natural candidates for realizing the variational principle because the measures outside of \(\mathcal{K}\) have small pressure.

Recall we assume that \(f\) has \emph{bounded unstable Jacobian distortion}: there exists \(\kappa\in(0,c_0/2)\) such that for every \(x,y\in \Lambda\),
\[
e^{-\kappa}
<
\frac{\left|\det\bigl(Df_x|_{E^u(x)}\bigr)\right|}
     {\left|\det\bigl(Df_y|_{E^u(y)}\bigr)\right|}
<
e^{\kappa}.
\]
Define
\[
J^u(x):=\log\left|\det\bigl(Df(x)|_{E^u(x)}\bigr)\right|.
\]
Then for an ergodic invariant measure $\mu$, we can write the sum of all unstable Lyapunov exponents in terms of $J^u(x)$: by Birkhoff's ergodic theorem, 
\[
\sum_{k=1}^{\dim E^u} \lambda^u(\mu) = \int_M\log\left|\det\bigl(Df(x)|_{E^u(x)}\bigr)\right|\,d\mu = \int_MJ^u(x)\,d\mu
\]

Let \(\mu_{\mathrm{SRB}}\) be an ergodic SRB measure whose center Lyapunov exponent \(\lambda^c(\mu_{\mathrm{SRB}})\) is at least \(c_0\).

\begin{lemma}\label{low_pressure_outside}
Let $\varphi: M\rightarrow\R$ be a continuous potential of bounded variation: \(\sup \varphi-\inf \varphi < \frac{c_0}{2}-\kappa\). Then for every \(\eta\notin\mathcal K\),
\[
h_{\eta}(f)+\int \varphi\,d\eta \le P_{\mathrm{top}}(\varphi;f)-\frac{c_0}{6}.
\]
In particular,
\[
\sup_{\eta\in\mathcal M_e(f)\setminus\mathcal K} \left\{h_{\eta}(f)+\int \varphi\,d\eta\right\}
\le P_{\mathrm{top}}(\varphi;f)-\frac{c_0}{6}.
\]
\end{lemma}

\begin{proof}
We compare \(h_{\eta}(f) + \int \varphi \, d\eta\) with
\(h_{\mu_{\mathrm{SRB}}}(f) + \int \varphi \, d\mu_{\mathrm{SRB}}\).
Since \(\eta \notin \mathcal{K}\), it follows that \(\lambda^c(\eta) \le \frac{c_0}{3}\). Therefore,
\begin{align*}
    h_{\eta}(f) + \int \varphi \, d\eta
    &\le \int J^u\,d\eta +\frac{c_0}{3}+ \int \varphi \, d\eta \\
    &\le \sup_{x\in\Lambda} J^u(x) +\frac{c_0}{3}
      + \int \varphi \, d\mu_{\mathrm{SRB}} + \bigl(\sup\varphi-\inf\varphi\bigr) \\
    &< \sup_{x\in\Lambda} J^u(x) +\frac{c_0}{3}
      +\left(\frac{c_0}{2}-\kappa\right)+ \int \varphi \, d\mu_{\mathrm{SRB}} \\
    &\le \inf_{x\in\Lambda} J^u(x)+\kappa+\frac{c_0}{3}
      +\left(\frac{c_0}{2}-\kappa\right)+ \int \varphi \, d\mu_{\mathrm{SRB}} \\
    &= \inf_{x\in\Lambda} J^u(x)+\frac{5c_0}{6}+\int \varphi \, d\mu_{\mathrm{SRB}} \\
    &\le h_{\mu_{\mathrm{SRB}}}(f) + \int \varphi \, d\mu_{\mathrm{SRB}} -\frac{c_0}{6}.
\end{align*}
Here, the first inequality follows from Ruelle's inequality in the setting of partial hyperbolicity with one-dimensional center:
\[
h_\eta(f)\le \int J^u\,d\eta + \lambda^c(\eta),
\]
and the fact that \(\lambda^c(\eta)\le c_0/3\). The second inequality uses
\[
\int J^u\,d\eta \le \sup_{x\in\Lambda} J^u(x)
\quad\text{and}\quad
\int\varphi\,d\eta \le \int\varphi\,d\mu_{\mathrm{SRB}}+(\sup\varphi-\inf\varphi).
\]
The strict inequality uses our assumption
\[
\sup \varphi - \inf \varphi < \frac{c_0}{2}-\kappa.
\]
The next inequality uses bounded unstable Jacobian distortion, which implies
\[
\sup_{x\in\Lambda} J^u(x)-\inf_{x\in\Lambda} J^u(x)<\kappa.
\]
The last inequality uses Pesin's entropy formula:
\[
h_{\mu_{\mathrm{SRB}}}(f)
=
\int J^u\,d\mu_{\mathrm{SRB}}+\lambda^c(\mu_{\mathrm{SRB}})
\ge
\inf_{x\in\Lambda} J^u(x) + c_0.
\]
By the variational principle, this proves that
\[
h_{\eta}(f) + \int \varphi \, d\eta\le P_{\mathrm{top}}(\varphi; f )-\frac{c_0}{6}.
\]
\end{proof}

For the one-dimensional dynamics considered in \cref{thm:main-1d}, we adapt \cref{low_pressure_outside} by replacing the small variation assumption with hyperbolic potential:

\begin{lemma}\label{low_pressure_outside_1d}
Assume the hypotheses of \cref{thm:main-1d}, and let \(\varphi\) be a H\"older continuous potential satisfying
\[
\max_{\nu\in\mathcal M(f|_J)}\int\varphi\,d\nu
<P_{\mathrm{top}}(\varphi;f|_J).
\]
After decreasing the expansion constant \(c_0>0\), if necessary, every invariant measure \(\mu\in\mathcal M(f|_J)\) with
\[
\int\log|f'|\,d\mu\le\frac{c_0}{3}
\]
satisfies
\[
h_\mu(f|_J)+\int\varphi\,d\mu
\le P_{\mathrm{top}}(\varphi;f|_J)-\frac{2c_0}{3}.
\]
\end{lemma}

\begin{proof}
The pressure hypothesis allows us to decrease \(c_0\) so that
\[
0<c_0\le P_{\mathrm{top}}(\varphi;f|_J)
-\max_{\nu\in\mathcal M(f|_J)}\int\varphi\,d\nu.
\]
The resulting constant may depend on \(\varphi\), and the non-uniform expansion hypothesis remains valid on the same set \(H\).

Topological exactness on \(J\) excludes strictly attracting periodic orbits. Hence every ergodic invariant measure on \(J\) has nonnegative Lyapunov exponent by \cite[Appendix A, Proposition A.1]{RL20}. Ruelle's inequality and ergodic decomposition therefore give
\[
h_\mu(f|_J)\le\int\log|f'|\,d\mu
\qquad\text{for every }\mu\in\mathcal M(f|_J).
\]
Consequently, if \(\int\log|f'|\,d\mu\le c_0/3\), then
\begin{align*}
h_\mu(f|_J)+\int\varphi\,d\mu
&\le\frac{c_0}{3}
+\max_{\nu\in\mathcal M(f|_J)}\int\varphi\,d\nu\\
&\le P_{\mathrm{top}}(\varphi;f|_J)-\frac{2c_0}{3}.
\end{align*}
\end{proof}

In the interval case, \cref{low_pressure_outside_1d} replaces
\cref{low_pressure_outside} and applies to all invariant measures,
including those in \eqref{eq:upper-bound-entropy}.

\subsubsection{Expanding Measures}
We show that every measure in $\mathcal{K}$ is an \emph{expanding measure} originally introduced in \cite{ABV00}. Let us recall its definition:

\begin{definition}
     A measure \(\mu\) is called a \(\sigma^{-1}\)-expanding measure if \(x \in H(\sigma)\) for \(\mu\)-almost every point \(x\). In other words, \(\mu(H(\sigma)) = 1\). Here \(H(\sigma)\) is the \(\sigma^{-1}\)-expanding set defined in \cref{expanding_set}.
\end{definition}

We first need to control the recurrence to the critical set $\mathcal{C}$ of generic points of invariant ergodic measures whose Lyapunov exponents are regular.

Recall that \hyperref[nonflat]{(S1)} in the non-flatness condition of the center direction near $\mathcal C$ says that there exist constants $B>1$ and $\beta>0$ such that for every $x\in M\setminus \mathcal C$,
\[
B^{-1}\, \dist(x,\mathcal C)^{\beta}
\;\le\;
\|Df(x)|_{E^c(x)}\|
\;\le\;
B\, \mathrm{dist}(x, \mathcal{C})^{\beta}.
\]

\begin{lemma}\label{dist_integrable}
    Let $\eta$ be an ergodic measure such that $\lambda^c(\eta) > -\infty$, where $\lambda^c$ is the Lyapunov exponent of $\eta$ associated to $E^c$. Then,
    \[
    \int |\log \, \dist (x, \mathcal{C})| \,d\eta < \infty.
    \]
\end{lemma}

\begin{proof}
    Since $\eta$ is ergodic,
    \[
    \lambda^c(\eta)=\int \log\|Df|_{E^c}\|\,d\eta > -\infty.
    \]
    By the non-flatness assumption,
    \[
    -\log \dist(x,\mathcal C)
    \;\le\;
    \frac{1}{\beta}\bigl(-\log \|Df(x)|_{E^c(x)}\|+\log B\bigr)
    \qquad\text{for every }x\in M\setminus \mathcal C.
    \]
    Hence $-\log \dist(x,\mathcal C)$ is $\eta$-integrable. Since $M$ is compact, $|\log \dist(x,\mathcal C)|$ differs from $-\log \dist(x,\mathcal C)$ by a bounded term, and the conclusion follows.
\end{proof}

\begin{lemma}\label{slow_rec}
    Let \(\eta\) be an ergodic measure with \(\lambda^c(\eta) > -\infty\). Then, for any \(\gamma > 0\), there exists \(\delta > 0\) such that for \(\eta\)-almost every point \(x\),
    \[
    \limsup_{n \to \infty}
    \frac{1}{n} \sum_{j=0}^{n-1}
    \bigl(-\!\log \, \dist_\delta\bigl(f^j(x), \mathcal{C}\bigr)\bigr)
    \;\le\;
    \gamma.
    \]
\end{lemma}

\begin{proof}
    By \cref{dist_integrable}, the function \(\bigl|\log \, \dist\bigl(x, \mathcal{C}\bigr)\bigr|\) is \(\eta\)-integrable, and in particular, \(\eta(\mathcal{C}) = 0\). Using these facts together with the definition of the \(\delta\)-truncated distance \(\dist_\delta\), we deduce that for any \(\gamma > 0\), there exists \(\delta > 0\) such that
    \[
    \int
    -\log \, \dist_\delta\bigl(x,\mathcal{C}\bigr)
    \, d\eta
    \;\le\;
    \gamma.
    \]
    Because \(\eta\) is ergodic, an application of Birkhoff’s Ergodic Theorem to the function \(-\log \, \dist_\delta\bigl(x, \mathcal{C}\bigr)\) completes the proof.
\end{proof}

We now show that there exists a single \(\sigma \in (0,1)\) such that every \(\mu \in \mathcal{K}\) is \(\sigma^{-1}\)-expanding. As will become clear later, a uniform choice of \(\sigma\) across all \(\mu \in \mathcal{K}\) is necessary for establishing that every ergodic measure in \(\mathcal{K}\) is almost expansive at the same scale.

\begin{lemma}
    Any \(\mu \in \mathcal{K}\) is \(\sigma^{-1}\)-expanding with \(\sigma = \exp\left(-\frac{c_0}{9}\right)\).
\end{lemma}

\begin{proof}
    Since \(\mu\) is ergodic, we have
    \[
    \lambda^c(x) = \lambda^c(\mu) = \int \log\|Df(x)|_{E^c}\| \, d\mu > \frac{c_0}{3}
    \]
    for \(\mu\)-almost every \(x\).

    By dominated splitting and uniform expansion along \(E^u\), after replacing the ambient metric by an adapted one if necessary, we may assume that
    \[
    \|Df(x)^{-1}\|=\|Df(x)|_{E^c(x)}^{-1}\|
    \]
    for every \(x \notin \mathcal C\). Hence, by Birkhoff's ergodic theorem,
    \[
    \lim_{n \to \infty} \frac{1}{n} \sum_{j=0}^{n-1} \log \|Df(f^j(x))^{-1}\|
    =
    - \lambda^c(\mu)
    <
    -\frac{c_0}{3}
    =
    3\log \sigma
    \]
    for \(\mu\)-almost every \(x\), where \(\sigma = \exp(-c_0/9)\).

    This fact, combined with \cref{slow_rec}, gives the assumptions needed for \(H(\sigma)\) at \(\mu\)-almost every point.
\end{proof}

\subsubsection{\texorpdfstring{Expansivity Inside \(\mathcal{K}\)}{Expansivity Inside K}}
Finally, we conclude the proof that any \(\sigma^{-1}\)-expanding measure \(\mu\) is almost expansive at scale \(\delta_1\), which implies every \(\mu\in\mathcal{K}\) is almost expansive.

\begin{lemma}\label{expansivity}
    For any \(\sigma^{-1}\)-expanding measure \(\mu\), there exists \(\delta_1 > 0\) such that for any \(\epsilon < \delta_1\) and for \(\mu\)-almost every \(x\),
    \[
    \Gamma_{\epsilon}(x) = \{x\}.
    \]
\end{lemma}

\begin{proof}
    \cref{lem:infty_hyp_time} guarantees that for \(\mu\)-almost every \(x\), there are infinitely many hyperbolic times \(n_i(x)\), since \(\mu\) is \(\sigma^{-1}\)-expanding. Therefore, applying \cref{hyp_times_prop}, we conclude that there exists some \(\delta_1 > 0\) such that if \(y \in \Gamma_{\epsilon}(x)\) with \(\epsilon < \delta_1\), then for any \(n_i\), we have
    \[
    \dist(x, y) \leq \sigma^{n_i/2} \, \dist(f^{n_i}(x), f^{n_i}(y)) \leq \sigma^{n_i/2} \epsilon.
    \]
    Since \(n_i \to \infty\) as \(i \to \infty\), we deduce that \(x = y\).
\end{proof}

\subsubsection{Pressure Gap}
\begin{theorem}\label{expansivity-gap}
We have the pressure gap for the set of non-expansive points
\[
P_{\mathrm{exp}}^{\perp}(\varphi) \;<\; P_{\mathrm{top}}(\varphi; f ).
\]
\end{theorem}

\begin{proof}
By definition, for any \(\epsilon < \delta_1\),
\[
P_{\mathrm{exp}}^{\perp}(\varphi,\epsilon)
\;=\;
\sup_{\mu \in \mathcal{M}_e(f)}
\Bigl\{
h_\mu(f) \;+\; \int \varphi \, d\mu
\;:\;
\mu\bigl(\mathrm{NE}(\epsilon)\bigr) > 0
\Bigr\}.
\]
Consider two cases for a measure \(\mu\):

\smallskip
\noindent
\textbf{Case 1:} \(\mu \notin \mathcal{K}\).
By \cref{low_pressure_outside} (or the interval pressure estimate following it for \cref{thm:main-1d}), we have a uniform bound:
\[
h_{\mu}(f)+\int \varphi\,d\mu
\;<\;
P_{\mathrm{top}}(\varphi; f )- \frac{c_0}{6}.
\]
In particular,
\[
\sup_{\mu\in\mathcal M_e(f)\setminus\mathcal K}
\left\{h_{\mu}(f)+\int \varphi\,d\mu\right\}
\le P_{\mathrm{top}}(\varphi;f)-\frac{c_0}{6}.
\]
Hence, such measures cannot achieve pressure \(P_{\mathrm{top}}(\varphi; f )\).

\smallskip
\noindent
\textbf{Case 2:} \(\mu \in \mathcal{K}\).
By \cref{expansivity}, there exists \(\delta_1 > 0\) so that for all \(\epsilon < \delta_1\),
\[
\mu\bigl(\mathrm{NE}(\epsilon)\bigr)
\;=\;
0.
\]
Thus, for sufficiently small \(\epsilon\), these measures place no mass on \(\mathrm{NE}(\epsilon)\) and do not contribute to \(P_{\mathrm{exp}}^{\perp}(\varphi,\epsilon)\).

In summary, for every \(\epsilon < \delta_1\), every measure that contributes to \(P_{\mathrm{exp}}^{\perp}(\varphi,\epsilon)\) lies outside \(\mathcal{K}\), and each such measure has pressure below \(P_{\mathrm{top}}(\varphi; f )\). Consequently,
\[
P_{\mathrm{exp}}^{\perp}(\varphi)
\;=\;
\lim_{\epsilon\rightarrow 0} P_{\mathrm{exp}}^{\perp}(\varphi,\epsilon)
\;\le\;
P_{\mathrm{top}}(\varphi;f)-\frac{c_0}{6}
\;<\;
P_{\mathrm{top}}(\varphi; f ).
\]
This completes the proof.
\end{proof}

\subsection{Non-uniform Specification and Regularity}\label{sec:spec}
In this section, we first give an explicit description of how to decompose the collection of all orbit segments \((x,n)\) into \((\mathcal{P}, \mathcal{G},\mathcal{S})\). We show that \(\mathcal{G}\) has the specification property at any scale \(\varepsilon\), and that orbits not in \(\mathcal{G}\) carry small pressure. In addition, any H\"older continuous \(\varphi\) has the Bowen property on \(\mathcal{G}\) at some fixed scale \(\varepsilon_0\).

\subsubsection{Construction of the Orbit Decomposition}
Let \(\psi^c(x) = \log \|Df|_{E^c(x)}\|\), so that \(\frac{1}{n} S_n \psi^c(x)\) represents the average expansion (or contraction) undergone by an orbit segment \((x, n)\) in the center direction. Additionally, given \(\mu\in \mathcal{M}_e(f)\), \(\lambda^c(\mu) =\int \psi^c\, d\mu\) is the center Lyapunov exponent of \(\mu\).

To define the orbit decomposition (\cref{fig}), for some small \(r>0\) we consider the following collection of orbit segments:
\begin{equation}\label{orb-decomp}
\begin{split}
    \mathcal{P} &:=\emptyset, \\
    \mathcal{G} &:=\{(x,n)\in \Lambda\times \N:\,S_{k}\psi^c(f^{n-k}x) > kr,\  \forall\, 0< k\leq n\},\\
    \mathcal{S} &:=\{(x,n)\in \Lambda\times \N:\, S_n\psi^c(x) \le nr\}.
\end{split}
\end{equation}

The collection \(\mathcal{G}\) is chosen so that along orbit segments from \(\mathcal{G}\), the center direction is uniformly backward contracting. In addition, \(f\) is locally invertible along these ``good'' orbits.

Together with the trivial collection \(\{(x,0) : x \in \Lambda\}\) for \(\mathcal{P}\), these collections define a decomposition of any orbit \((x, n) \in \Lambda \times \mathbb{N}\) as follows: let \(s\) be the largest integer in \(\{0, \cdots, n\}\) such that \( S_{s}\psi^c(f^{n-s}x) \le sr\), meaning \((f^{n-s}x, s) \in \mathcal{S}\).

We must then have \((x,  n-s) \in \mathcal{G}\), since if \(S_{k_0}\psi^c(f^{n-s-k_0} x) \le k_0r\) for some \(0 < k_0 \leq n-s\), then
\[
S_{s+k_0}\psi^c(f^{n-s-k_0}x) = S_{k_0}\psi^c(f^{n-s-k_0}x) + S_{s}\psi^c(f^{n-s}x) \le (s+k_0)r,
\]
which would contradict the maximality of \(s\).

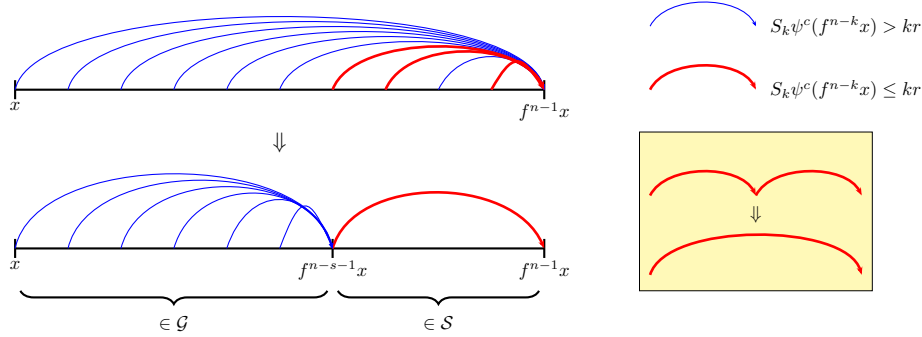
\begin{figure}[ht]
  \centering
  \begin{tikzpicture}[thick,scale=0.7, transform shape]  % add ,scale=.8,transform shape if desired
% --------------------------------------------------------------
% TOP orbit :  x ---------------→ f^n(x)
% --------------------------------------------------------------
\coordinate (x0) at (0,0);
\coordinate (xn) at (10,0);
\draw (x0) -- (xn);

% end ticks + labels
\foreach \p/\lab in {x0/$x$,xn/$f^{n-1}x$}{
  \draw (\p) ++(0,-.18) -- ++(0,.36);
  \node[below=3pt] at (\p) {\lab};
}

% blue “good’’ arcs (all k)
\def\flat{0.7}  % 0<\flat<1  → flatter;  1 = original bumpiness

\foreach \start [count=\k from 0] in {0,...,9}{%
  \pgfmathsetmacro{\h}{\flat*(1 + 0.18*(9-\k))}
  \draw[blue,->,line width=.4pt]
        (\start,0)
        .. controls +(.4,\h) and +(-.45,\h) ..
        (10,0);
}

% red “bad’’ arcs (subset)
\foreach \start in {9,7,6}{%
  \pgfmathsetmacro{\h}{\flat*(1 + 0.18*(9-\start))}  % identical height as blue arc
  \draw[red,->,line width=1pt]              % drawn AFTER the blue arcs
        (\start,0)
        .. controls +(.4,\h) and +(-.45,\h) ..
        (10,0);
}

% big double arrow
\node at ($(x0)!0.5!(xn)+(0,-1)$) {\Large$\Downarrow$};

% --------------------------------------------------------------
% LOWER orbit :  0 -- f^p(x) ----→ (good) ----→ t
% --------------------------------------------------------------
\coordinate (x1) at (0,-3);
\coordinate (xs) at (6,-3);
\coordinate (xn) at (10,-3);
\draw (x1 |- xs) -- (xn);

% ticks + labels
\foreach \p/\lab in {x1/$x$,xs/$f^{n-s-1}x$,xn/$f^{n-1}x$}{
  \draw (\p) ++(0,-.18) -- ++(0,.36);
  \node[below=3pt] at (\p) {\lab};
}

% red arc (bad prefix)
\draw[red,->,line width=1pt]
      (xs |- xn) .. controls +(.4,1.4) and +(-.6,1.4) .. (xn);

% blue family (good core)
% choose once:
\def\flat{0.9}  % 0<\flat<1  → flatter;  1 = original bumpiness

\foreach \start in {0,1,2,3,4,5}{%
  \pgfmathsetmacro{\h}{\flat*(1+0.18*(6-\start)}    
  \draw[blue,->,line width=.4pt]
        (\start,-3) 
        .. controls +(.4,\h) and +(-.45,\h) ..
        (6,-3);
}

% braces + set labels
\draw[decorate,decoration={brace,mirror,amplitude=5pt}]
      ($(x1 |- xs)+(0.1,-0.8)$) -- ($(xs)+(-0.1,-0.8)$)
      node[midway,below=10pt] {$\in \mathcal{G}$};

\draw[decorate,decoration={brace,mirror,amplitude=5pt}]
      ($(xs)+(0.1,-0.8)$) -- ($(xn)+(-0.1,-0.8)$)
      node[midway,below=10pt] {$\in \mathcal{S}$};

% --------------------------------------------------------------
% RIGHT-HAND legend
% --------------------------------------------------------------
% exemplar blue and red arcs with inequalities
\coordinate (legB0) at (12,1.2);
\coordinate (legB1) at (14,1.2);
\draw[blue,->, line width=0.4pt] (legB0) .. controls +(0.3,.6) and +(-0.3,.6) .. (legB1);
\node[right=4pt] at (legB1) {$S_{k}\psi^c(f^{n-k}x)> kr$};

\coordinate (legR0) at (12,0);
\coordinate (legR1) at (14,0);
\draw[red,->,line width=1pt] (legR0) .. controls +(0.3,.6) and +(-.3,.6) .. (legR1);
\node[right=4pt] at (legR1) {$S_{k}\psi^c(f^{n-k}x)\le kr$};

% yellow “collapsed’’ picture
\begin{scope}[shift={(12,-2)}]
  \fill[yellow!35] (-0.2,-1.8) rectangle (4.2,1.2);
  \draw[black, line width=.2pt] (-0.2,-1.8) rectangle (4.2,1.2);

  % two red hops
  \foreach \x/\h in {0/1.0,2/1.0}{
    \draw[red,->,line width=1pt]
          (\x,0) .. controls +(0.3,.6) and +(-0.3,.6) .. (\x+2,0);
  }

  % downward arrow
  \node at (2,-0.35) {$\Downarrow$};

  % long red hop
  \draw[red,->,line width=1pt]
        (0,-1.5) .. controls +(0.3,1) and +(-0.3,1) .. (4,-1.5);
\end{scope}

% --------------------------------------------------------------
\end{tikzpicture}
  \caption{Orbit decomposition}
  \label{fig}
\end{figure}

\subsubsection{Specification on \texorpdfstring{$\mathcal{G}$}{G}}
To prove specification on \(\mathcal{G}\), we use backward contraction along good orbit segments together with topological exactness. More precisely, since \(f|_\Lambda\) is topologically exact, so for every \(\rho>0\), there exists \(\tau=\tau(\rho)\in\mathbb N\) such that for every \(x\in \Lambda\),
\[
f^j(B(x,\rho))=\Lambda
\qquad\text{for some }0\le j\le \tau.
\]

We begin by establishing that every ``good'' orbit in \(\mathcal{G}\) exhibits backward uniform contraction:

\begin{lemma}\label{back_contract}
There exist $\varepsilon_0>0$ and $C\ge 1$ such that for every $(x,n)\in\mathcal G$ and every $y\in B_n(x,\varepsilon_0)$ we have
\begin{equation}\label{eq:backward-contraction}
d\bigl(f^k x, f^k y\bigr)
\;<\;
C e^{-r(n-1-k)}\, d\bigl(f^{n-1}x, f^{n-1}y\bigr)
\;<\;
C e^{-r(n-1-k)}\,\varepsilon_0
\end{equation}
for all $0\le k\le n-1$.
\end{lemma}

\begin{proof}
From $(x,n)\in\mathcal G$ and the definition of $\psi^c$ we obtain a uniform backward contraction along the center:
\[
\bigl\|Df^{-(n-1-k)}\big|_{E^c(f^{n-1} x)}\bigr\|
< e^{-r(n-1-k)},\qquad 0\le k\le n-1.
\]
Uniform expansion along $E^u$ gives
\[
\bigl\|Df^{-(n-1-k)}\big|_{E^u(f^{n-1} x)}\bigr\|
< e^{-r(n-1-k)}.
\]
Since the system is partially hyperbolic, the $Df$-invariant splitting
$TM = E^u \oplus E^c$ is continuous and dominated, with $\dim E^c = 1$. In particular, the angle between $E^u$ and $E^c$ is uniformly bounded away from zero, so there is a constant $C\ge1$ such that
\[
\bigl\|Df^{-(n-1-k)}(z)\bigr\|
< C e^{-r(n-1-k)}.
\]
for every point $z$ in a small neighborhood of $f^{n-1} x$ and every $0\le k\le n-1$. Shrinking this neighborhood if needed and using compactness, we may choose $\varepsilon_0>0$ so that whenever $d(z,f^{n-1} x)<\varepsilon_0$ the above estimate holds for all such $k$.

Now take $y\in B_n(x,\varepsilon_0)$. Then
$d(f^{n-1} x,f^{n-1} y)<\varepsilon_0$. Let $\gamma$ be a minimizing geodesic from $f^{n-1} x$ to $f^{n-1} y$; it is contained in the $\varepsilon_0$-ball around $f^{n-1} x$. For each $0\le k\le n-1$ we have
\begin{align*}
    d\bigl(f^k x,f^k y\bigr)
&= d\bigl(f^{-(n-1-k)} f^{n-1} x, f^{-(n-1-k)} f^{n-1} y\bigr) \\
&< \sup_{w\in\gamma} \bigl\|Df^{-(n-1-k)}(w)\bigr\|
\cdot d\bigl(f^{n-1} x,f^{n-1} y\bigr),
\end{align*}
and the supremum is bounded by $C e^{-r(n-1-k)}$. This gives \eqref{eq:backward-contraction}.
\end{proof}

\begin{theorem}\label{specification}
    The collection \(\mathcal{G}\) has specification at any scale \(\varepsilon>0\).
\end{theorem}

\begin{proof}
Since having specification at some scale \(\varepsilon_*>0\) implies specification at every larger scale \(\varepsilon\ge \varepsilon_*\), it is enough to prove the theorem for all sufficiently small \(\varepsilon>0\). We assume from now on that \(0<\varepsilon\le \varepsilon_0\), where \(\varepsilon_0\) is given by Lemma~\ref{back_contract}.

Let \(\left\{(x_i, n_i)\right\}_{i=1}^k\) be any finite collection of orbit segments from \(\mathcal G\). Since $(x_i,n_i)\in\mathcal G$, \cref{back_contract} gives backward uniform contraction along this ``good'' orbit: for all $y\in B_{n_i}(x_i,\varepsilon)$ and $0\le m\le n_i-1$,
\[
d\bigl(f^m x_i, f^m y\bigr)
\;<\; C e^{-r(n_i-1-m)}\,
d\bigl(f^{n_i-1}x_i, f^{n_i-1}y\bigr)
\;<\; C e^{-r(n_i-1-m)}\varepsilon.
\]
Thus the inverse branch of $f^{n_i-1}$ that maps back to $x_i$ is uniformly contracting on a neighborhood of $f^{n_i-1}x_i$, with contraction rate $e^{-r}$ up to a uniform multiplicative constant. In particular, $f^{n_i-1}$ maps $B_{n_i}(x_i,\varepsilon)$ diffeomorphically onto an open neighborhood of $f^{n_i-1}x_i$ that contains the open disk $B(f^{n_i-1}x_i,\varepsilon/C)$.

By topological exactness, for any \(\varepsilon >0 \), there exists some \(\tau = \tau(\varepsilon/C) \in \mathbb{N}\) such that for any open disk \(B(x,\varepsilon/C)\), we have \(f^j(B(x, \varepsilon/C)) = \Lambda\) for some \(j \leq \tau\). Hence for any \(i = 1, \ldots, k-1\), we can select \(\tau_i \leq \tau(\varepsilon/C)\) such that the image of \(B(f^{n_i-1} x_i, \varepsilon/C)\) under \(f^{\tau_i}\) covers \(B_{n_{i+1}}(x_{i+1}, \varepsilon)\).

As a consequence of these, we define
\[
t_1=0 \quad \text{and} \quad
t_j=\sum_{i=1}^{j-1} (n_i-1 + \tau_i)  \quad \text{ for } j=2, \ldots, k
\]
and set
\[
U_1 := B_{n_1}(x_1, \varepsilon).
\]
For \(j=1,\ldots,k-1\), define
\[
U_{j+1}
:= \{u\in U_j : f^{t_{j+1}}u \in B_{n_{j+1}}(x_{j+1},\varepsilon)\}.
\]
By definition each \(U_j\) is open, and the family is nested:
\[
U_k \subset \cdots \subset U_{j+1} \subset U_j \subset \cdots \subset U_1.
\]

We claim that for every \(j=1,\ldots,k\),
\begin{enumerate}
\item \(U_j\neq\emptyset\),
\item \(f^{t_j}(U_j)=B_{n_j}(x_j,\varepsilon)\).
\end{enumerate}

For \(j=1\) this is clear, since \(t_1=0\) and
\(U_1 = B_{n_1}(x_1,\varepsilon)\).

Assume that (1)–(2) hold for some \(1\le j\le k-1\). Let \(z\in B_{n_{j+1}}(x_{j+1},\varepsilon)\) be arbitrary. By the choice of \(\tau_j\), there exists
\[
w \in B\bigl(f^{n_j-1}x_j,\varepsilon/C\bigr) \subset W_j = f^{n_j-1}\bigl(B_{n_j}(x_j,\varepsilon)\bigr)
\]
such that \(f^{\tau_j}w = z\). Since \(w\in W_j\), there is a unique \(y\in B_{n_j}(x_j,\varepsilon)\) with \(f^{n_j-1}y = w\). By the induction hypothesis \(f^{t_j}(U_j)=B_{n_j}(x_j,\varepsilon)\), so we can choose \(u\in U_j\) with \(f^{t_j}u = y\). Then
\[
f^{t_{j+1}}u
= f^{t_j+n_j-1+\tau_j}u
= f^{\tau_j}w
= z\in B_{n_{j+1}}(x_{j+1},\varepsilon),
\]
so \(u\in U_{j+1}\). Since \(z\) was arbitrary, this shows that
\(f^{t_{j+1}}(U_{j+1}) = B_{n_{j+1}}(x_{j+1},\varepsilon)\), and in
particular \(U_{j+1}\neq\emptyset\). This completes the induction.

Thus all \(U_j\) are nonempty and
\[
U_k \subset \cdots \subset U_2 \subset U_1,
\]
so
\[
\bigcap_{j=1}^{k} U_j = U_k \neq \emptyset.
\]
Choose any \(y \in \bigcap_{j=1}^{k} U_j\). Then for every \(1\le j\le k\)
we have \(f^{t_j}(y) \in B_{n_j}(x_j,\varepsilon)\), which is equivalent to
\[
d_{n_j}\bigl(f^{t_{j}}(y),x_j\bigr)<\varepsilon
\quad \text{for every }1\le j\le k.
\]
Thus the orbit of \(y\) \(\varepsilon\)–shadows the collection of orbit segments \(\left\{(x_j,n_j)\right\}_{j=1}^k\) with gaps bounded by \(\tau(\varepsilon/C)\), and \(\mathcal G\) has specification at scale \(\varepsilon\).
\end{proof}

\subsubsection{Pressure Gap}
We prove that if \(0<r\leq \frac{c_0}{3}\), then for potentials \(\varphi\) satisfying the hypothesis of \cref{thm:main}, or of \cref{thm:main-1d} with the decreased \(c_0\) specified above, the ``bad orbits'' carry less topological pressure than the whole space. Recall we define \(\psi^c(x) = \log \|Df|_{E^c(x)}\|\) and
\[
\mathcal{S} \;:=\;
\bigl\{ (x, n) \in \Lambda \times \mathbb{N} : S_n \psi^c(x) \leq nr \bigr\}.
\]
We have the following pressure gap estimate:
\begin{theorem}\label{specification-gap}
    \[
    P(\mathcal{P} \cup \mathcal{S}, \varphi)< P_{\mathrm{top}}(\varphi; f ).
    \]
\end{theorem}

\begin{proof}
    We start by taking \(\mathcal{P} = \emptyset\). To describe \(P(\mathcal{S},\varphi)\), recall we have proven in \cref{low_pressure_outside} (using the all-invariant interval estimate following it in the proof of \cref{thm:main-1d}) that
    \[
    \sup_{\mu \in \mathcal{M}_f} \Bigl\{ h_{\mu}(f) +\int \varphi \,d\mu:   \lambda^c(\mu) \leq \frac{c_0}{3} \Bigr\}
    < P_{\mathrm{top}}(\varphi; f ).
    \]
    So for \(0<r\leq \frac{c_0}{3}\),
    \begin{equation}\label{eq:upper-bound-entropy}
    \sup_{\mu \in \mathcal{M}(f)} \Bigl\{ h_{\mu}(f) +\int \varphi \,d\mu:   \lambda^c(\mu) \leq r \Bigr\}
    < P_{\mathrm{top}}(\varphi; f ).
    \end{equation}

    Since we define \(\psi^c(x) = \log \| Df|_{E^c(x)} \|\), for an \(f\)-invariant measure \(\mu\), its center Lyapunov exponent satisfies
    \(\lambda^c(\mu) = \int \psi^c \, d\mu\).

    The left-hand side of \eqref{eq:upper-bound-entropy} can be related to \(P(\mathcal{S},\varphi)\): Let \(E_n \subset \mathcal{S}_n\) be any \((n,\varepsilon)\)-separated set. Consider the Borel probability measure

    \[
    \nu_n \;=\;\frac{1}{\sum_{x\in E_n}e^{S_n\varphi(x)}}\, \sum_{x\in E_n}\delta_x e^{S_n\varphi(x)}, \quad
    \mu_n \;=\; \frac{1}{n}\sum_{k=0}^{n-1} f_*^k \nu_n.
    \]

    Then half of the proof of the variational principle (see \cite[Theorem 8.6]{Wal82}) shows that if \(\mu\) is a limit point of \(\mu_n\), then \(\mu\) is \(f\)-invariant and satisfies
    \[
    h_{\mu}(f) +\int \varphi\, d\mu\;\geq\; P(\mathcal{S}, \varphi, \varepsilon).
    \]
    Moreover, by weak\(^*\)-convergence and the definition of \(\mathcal{S}\), we have
    \[
    \lambda^c(\mu) \;=\; \int \psi^c \, d\mu \;\leq\; r.
    \]
    Combining this with \eqref{eq:upper-bound-entropy}, we conclude that
    \[
    P(\mathcal{S}, \varphi) \;<\; P_{\mathrm{top}}(\varphi; f ).
    \]
    Since \(\mathcal{P} = \emptyset\), this implies \(P(\mathcal{P} \cup \mathcal{S}, \varphi) < P_{\mathrm{top}}(\varphi; f )\). By definition, \(P_{\mathrm{spec}}^{\perp}(\varphi)<P_{\mathrm{top}}(\varphi; f )\), completing the proof.
\end{proof}

\subsubsection{Bowen Property}

\begin{theorem}\label{bowen-property}
Any H\"older continuous potential \(\varphi\) satisfies the Bowen property on \(\mathcal{G}\) at scale \(\varepsilon_0\).
\end{theorem}

\begin{proof}
Because \(\varphi\) is H\"older, there exist constants \(K>0\) and \(\alpha \in (0,1)\) such that
\[
\bigl|\varphi(x) - \varphi(y)\bigr|
\;\le\;
K\,\bigl(d(x,y)\bigr)^\alpha
\quad
\text{for all } x,y \in \Lambda.
\]
Take \((x,n)\in \mathcal{G}\) and \(y \in B_n(x,\varepsilon)\). By \cref{back_contract},
\[
\bigl|S_n \varphi(x) \;-\; S_n \varphi(y)\bigr|
\;\le\;
K \sum_{k=0}^{n-1}
\Bigl(d\bigl(f^k x,\, f^k y\bigr)\Bigr)^\alpha
\;<\;
K \sum_{k=0}^{n-1}
\Bigl(Ce^{-r\,(n-1-k)}\,\varepsilon\Bigr)^\alpha.
\]
Re-indexing with \(j = n-1-k\) and noting the resulting geometric series,
\[
K\sum_{k=0}^{n-1}
\Bigl(Ce^{-r(n-1-k)}\,\varepsilon\Bigr)^\alpha
\;<\;
KC^{\alpha}\,\varepsilon^\alpha
\sum_{j=0}^{\infty} e^{-r\,\alpha\,j}
\;<\;
\infty.
\]
Hence, \(\bigl|S_n\varphi(x) - S_n\varphi(y)\bigr|\) is bounded uniformly in \(x,n,\) and \(y\), establishing the Bowen property at scale \(\varepsilon\).
\end{proof}

\subsection{Verification of the Climenhaga--Thompson Criteria}\label{sec:verif}
\begin{proof}[Proof of \cref{thm:main}]We now complete the proof that if a system  satisfies the hypotheses of \cref{thm:main}, then the assumptions of \cref{ct} hold. Consequently, there exists a unique equilibrium state \(\mu\) for \((\Lambda,f,\varphi)\). Moreover, \(\mu\) satisfies the upper level-2 large deviations bound of \cref{ct}.

By \cref{expansivity-gap}, \(P_{\mathrm{exp}^+}^{\perp}(\varphi) \;<\; P_{\mathrm{top}}(\varphi; f )\). We define the decomposition \((\mathcal{P}, \mathcal{G}, \mathcal{S})\) as in \cref{orb-decomp}. In \cref{specification}, we showed that \(\mathcal{G}\) has specification at all scales. In \cref{bowen-property}, we showed that \(\varphi\) has the Bowen property on \(\mathcal{G}\). By \cref{specification-gap}, \(P\bigl(\mathcal{P} \cup \mathcal{S}, \varphi\bigr) \;<\; P_{\mathrm{top}}(\varphi; f )\). Thus, we see that under the hypotheses of \cref{thm:main}, all the hypotheses of \cref{ct} are satisfied for the decomposition \((\mathcal{P}, \mathcal{G}, \mathcal{S})\).
\end{proof}

\begin{proof}[Proof of \cref{thm:main-1d}]
Let \(\varphi\) be a H\"older continuous potential satisfying the pressure hypothesis. We decrease \(c_0\), if necessary, as in \cref{low_pressure_outside_1d}. In the one-dimensional setting, \(E^u\) is trivial and
\[
\lambda^c(\mu)=\int\log|f'|\,d\mu.
\]
As above, we let \(\mathcal K\) be the collection of ergodic invariant measures satisfying \(\lambda^c(\mu)>c_0/3\), and take \(\sigma=e^{-c_0/9}\). By \cref{low_pressure_outside_1d}, for every ergodic invariant measure \(\mu\notin\mathcal K\), we have
\[
h_\mu(f|_J)+\int\varphi\,d\mu
\le P_{\mathrm{top}}(\varphi;f|_J)-\frac{2c_0}{3}.
\]
The argument in \cref{expansivity-gap} therefore applies with \cref{low_pressure_outside_1d} in place of \cref{low_pressure_outside}.

We define the decomposition \((\mathcal P,\mathcal G,\mathcal S)\) as in \eqref{orb-decomp}, with \(0<r\le c_0/3\). Since \cref{low_pressure_outside_1d} applies to every invariant measure, it also gives \eqref{eq:upper-bound-entropy}. Thus, the argument in \cref{specification-gap} gives
\[
P(\mathcal P\cup\mathcal S,\varphi)
<P_{\mathrm{top}}(\varphi;f|_J).
\]
The proofs of \cref{specification,bowen-property} apply without change: \(\mathcal G\) has specification at all scales, and \(\varphi\) has the Bowen property on \(\mathcal G\). 

Consequently, all the hypotheses of \cref{ct} are satisfied. There exists a unique equilibrium state \(\mu_\varphi\), and it satisfies the stated upper large deviations bound. Finally,
\[
\begin{aligned}
h_{\mu_\varphi}(f|_J)
&=P_{\mathrm{top}}(\varphi;f|_J)
  -\int\varphi\,d\mu_\varphi\\
&\ge P_{\mathrm{top}}(\varphi;f|_J)
  -\max_{\nu\in\mathcal M(f|_J)}\int\varphi\,d\nu
>0.
\end{aligned}
\]
\end{proof}

\section{Applications}
In this section, we introduce the class of Viana maps, first constructed by Viana in \cite{Via97}. We review their main properties and verify that they satisfy all the hypotheses of our main theorem. It follows that Viana maps admit a unique equilibrium state for all sufficiently small potentials. Since the present paper is formulated in a general setting, we refer the reader to \cite{Li26} for detailed proofs specialized to Viana maps.

Viana maps form a fundamental class of non-uniformly expanding surface endomorphisms, realized as skew-products of an expanding circle map with a perturbed quadratic family on the fibers. Along the one-dimensional \emph{center} bundle, the derivative is typically expanding. However, when an orbit passes near the critical line $t=0$, the quadratic term folds the fiber, causing a sharp loss of expansion in the local derivative and thereby weakening the accumulated growth along the orbit. This intermittent loss of expansion prevents the system from being uniformly hyperbolic, while still allowing for a rich thermodynamic formalism.

Using our main theorem, we show that for Viana maps every H\"older continuous function whose oscillation is below an explicit threshold admits a unique equilibrium state. The proof adapts the Climenhaga--Thompson machinery: we carve the dynamics into a ``good'' core that enjoys specification and the Bowen property and a ``bad'' tail whose pressure we show is strictly lower than the topological pressure.

\subsection{Viana maps}
Before stating the main theorem, recall Viana’s construction. We refer to \cref{sec:viana-maps} for a more detailed discussion of the dynamical properties of Viana maps.
\begin{definition}[Viana map]
    Consider the map  
    \[
    f_{\alpha}(\theta,t)=\bigl(d\theta \mod 1,\;a_0-t^2+\alpha \sin (2\pi\theta)\bigr),
    \qquad
    (\theta,t)\in S^{1}\times \R,
    \]
Here \(d \geq 2\) is an integer, \(\alpha\) is a small real number and the parameter \(a_0 \in (1, 2)\) is such that the map \(h(t) = a_0 - t^2\) has a pre-periodic (but not periodic) critical point. For \(\alpha\) small enough there is an interval \(I \subset (-2,2)\) for which \(f_{\alpha}(S^1\times I) \subset \mathrm{int}(S^1\times I)\).
\end{definition}

In \cite{Via97}, Viana constructed this quadratic skew-products over the linear strongly expanding map of the circle and proved that if \(d\geq 16\), then there exists a positive constant $c_0 > 0$ such that any map $f$ sufficiently $C^{3}$-close to $f_{\alpha}$ has both Lyapounov exponents greater than $c_0$ Lebesgue-a.e.. Buzzi, Sester and Tsujii later in \cite{BST03} proved that if \(d\ge2\), then for every sufficiently small \(\alpha>0\) the same is true in a $C^{\infty}$-neighborhood.

We give a precise version of \cref{Li26} with explicit bounds on the variation of the potential. The main result of this section is the following.

\begin{theorem}[\cite{Li26}, Theorem A]\label{main}
    For all \(d\ge16\) and \(\epsilon>0\) sufficiently small, there is an \(\alpha_0>0\) such that for all \(\alpha\le\alpha_0\) and all Hölder potentials  with \(\sup \varphi-\inf \varphi <\frac{c_0}{2}-\log \frac{d+\epsilon}{d-\epsilon}\), the Viana map \(f_{\alpha}\) has a unique equilibrium state \(\mu\) for \(\varphi\). Moreover, \(\mu\) has the upper level-2 large deviations bound of \cref{ct}.
\end{theorem}

In \cite{Li26}, the proof was designed so that every step survives small perturbations to the dynamics. In particular, the expansivity, specification, and Bowen property we establish for \(f_{\alpha}\) remain valid for all maps in a sufficiently small \(C^{3}\) neighborhood. In the broader framework of this paper, this robustness follows directly from the fact that every such map $f$ satisfies the hypotheses of \cref{thm:main}. This robustness yields the following extension:

\begin{theorem}\label{main_B}
    The same result in \cref{main} is true for any sufficiently small \(C^{3}\) perturbation of \(f_{\alpha}\) (which is not necessarily a skew-product).
\end{theorem}

Henceforth we use the term \emph{Viana map} for any map \(f\) that lies in a sufficiently small \(C^{3}\)-neighborhood of the reference skew product \(f_{\alpha}\). Unless explicitly stated otherwise, every result we state for a Viana map therefore applies simultaneously to \(f_{\alpha}\) itself and to all its \(C^{3}\) perturbations within this neighborhood.

In this paper we study Viana maps on the forward-invariant region $S^1 \times I$
from Definition~1.1. For such a map $f$ we set
\[
  \Lambda := \bigcap_{n \ge 0} f^n(S^1 \times I),
\]
which equals $f^2(S^1 \times I)$ once $I$ is chosen as in Viana's construction.
All notions of pressure and equilibrium states in Theorems~A and~B are taken for
the restriction $f|_\Lambda : \Lambda \to \Lambda$, i.e.\ among $f$-invariant
Borel probability measures supported on $\Lambda \subset S^1 \times I$.

\subsubsection{Historical Remarks}
We briefly situate our results within the existing literature on Viana maps. Alves proved existence and uniqueness of the SRB measure \cite{Alv00}.
Recently, using the countable Markov coding approach, Lima, Obata and Poletti also obtained uniqueness of the measure of maximal entropy \cite{LOP26}. Alves, Oliveira and Santana employed an inducing scheme to establish existence and finiteness of the measure of maximal entropy \cite{AOS24}.

The present work complements these results in several ways: \cref{main} and \cref{main_B} give a coding free and robust proof of existence and uniqueness for all H\"older potentials whose oscillation is below an explicit threshold, in a form that is uniform over a whole $C^3$-neighborhood of the reference Viana map and hence applies to non-skew perturbations. In addition, the orbit decomposition we build yields a weak Gibbs estimate for the equilibrium state and an upper level-2 large deviation principle. Moreover, to the best of our knowledge this provides the first example of the orbit decomposition approach implemented in a non-invertible setting with critical points.

\subsubsection{Main Properties}\label{sec:viana-maps}
We begin by gathering the necessary materials for the construction of Viana maps with the desired decomposition properties.

We consider from here on these maps \(f\) close to \(f_{\alpha}\) restricted to \(S^1 \times I\). The main properties of \(f\) in a \(C^3\) neighborhood of \(f_{\alpha}\) are summarized below; see \cite{Via97}, \cite{Alv00}.
\begin{itemize}
    \item \textit{Non-uniformly expanding}, that is, for $\alpha$ sufficiently small, there exists a positive constant $c_0 > 0$ such that any map $f$ sufficiently close to $f_{\alpha}$ in the $C^3$ topology has both Lyapounov exponents greater than $c_0$ at Lebesgue almost every point, that is, for Lebesgue a.e. $x=(\theta, t) \in S^1 \times I$.
           
    \[\limsup_{n\rightarrow  +\infty} \frac{1}{n} \sum_{j=0}^{n-1}\log \|Df(f^jx)^{-1}\| < c_0;\]

    \item \textit{Slow recurrence to the critical set}, that is, letting
    \[
    \mathcal C := \{x \in S^1 \times I : \det|Df_x|=0\}
    \]
    denote the critical set of $f$, for every $\gamma>0$ there exists $\delta > 0$
    such that for Lebesgue a.e.\ $x \in S^1 \times I$,
    \[
     \limsup_{n\rightarrow\infty} \frac{1}{n}\sum_{j=0}^{n-1}
     -\log \bigl[\mathrm{dist}_{\delta}(f^j(x),\mathcal{C})\bigr]\leq \gamma.
    \]

    \item Existence and uniqueness of an ergodic absolutely continuous invariant (SRB) measure \(\mu_{\mathrm{SRB}}\);

    \item The attractor of \(f\) within the forward-invariant region \(S^1 \times I\) is defined as the intersection
    \[
    \Lambda =\bigcap_{n\geq 0} f^n(S^1\times I)
    \]
    of all forward images of \(S^1 \times I\), which is just \(\Lambda = f^2(S^1 \times I)\), as long as the interval \(I\) is properly chosen; 

    \item 
    Topologically exact: for any open set \( O \subset S^1 \times I \), there exists \(n= n(O) \in \mathbb{N} \) such that 
    \[
    f^{n}(O) = \Lambda.
    \]
    We note that, in \cite{AV02}, Alves and Viana prove that Viana maps satisfy what they call \emph{strongly topologically mixing} property. This is the same property that we refer to as \emph{topological exactness}, and throughout we use the latter terminology.
\end{itemize}

Next, we study the dynamics of the tangent bundle for Viana maps. Denote by \( f_0 \) the direct product map 
\[
    f_0 = g \times h: S^1 \times I \to S^1 \times I, \quad (\theta, t) \mapsto (d\theta \mod 1, a_0 - t^2).
\]

Since \( Df_0 \cdot \frac{\partial}{\partial \theta} = d \cdot \frac{\partial}{\partial \theta} \) and \( Df_0 \cdot \frac{\partial}{\partial t} = -2t \cdot \frac{\partial}{\partial t} \), with \( d \geq 16 > 4 \) and \( I \subset (-2, 2) \), it is clear that \( f_0 \) is a partially hyperbolic map of the form \( T_{(\theta,t)}(S^1 \times I) = E^u \oplus E^c \), where \( E^u \) is uniformly expanding and \( E^c \) is dominated by \( E^u \). From the theory of partial hyperbolicity, we know that if \( \alpha \) is sufficiently small, the map \( f_{\alpha} \) (like any \( C^1 \)-nearby map) also admits a partially hyperbolic splitting of the form \( E^u \oplus E^c \), which varies continuously. As a consequence, we get

\begin{proposition}\label{unstable_Lyap_exp}
    Given $\epsilon > 0$, if $\alpha$ is sufficiently small, the Lyapounov exponent
    $$\lambda^u(x, f):= \lim_{n\rightarrow\infty}\frac{1}{n}\log \|Df^n|_{E^u(x)}\|$$
    associated to $E^u$ at every point $x$ satisfies
    $$\log(d-\epsilon)\leq \lambda^u(x, f) \leq \log(d+\epsilon),$$
    for any $f$ close to $f_{\alpha}$.
\end{proposition}

By applying the previous theorems we derive a lower bound on the entropy of the SRB measure \( \mu_{\mathrm{SRB}} \). By Pesin's formula, the entropy \( h_{\mu_{\mathrm{SRB}}}(f) \) of \( \mu_{\mathrm{SRB}} \) is given by the integrated sum of the positive Lyapunov exponents of \( \mu_{\mathrm{SRB}} \). Therefore, by \cref{unstable_Lyap_exp}, we obtain the following bound:
\[
h_{\mu_{\mathrm{SRB}}}(f) \geq \log(d - \epsilon) + c_0.
\]

\subsubsection{Existence and Uniqueness of Equilibrium States}
We now verify the hypotheses of \cref{thm:main} for Viana maps. As discussed in the previous sections, Viana maps admit a partially hyperbolic splitting
\[
T_\Lambda M = E^u \oplus E^c,
\qquad \dim E^c = 1,
\]
where \(E^u\) is uniformly expanding and the critical set is non-flat in the center direction. Moreover, Viana maps are non-uniformly expanding and satisfy slow recurrence to the critical set for Lebesgue almost every point. The attractor is also topologically exact, so every sufficiently small open set is mapped onto the whole attractor after finitely many iterates. Finally, the unstable Jacobian has uniformly bounded distortion, since the unstable direction is inherited from the expanding base map and its derivative remains uniformly close to the constant base expansion.

Therefore all the hypotheses of \cref{thm:main} are satisfied by Viana maps. We conclude that every H\"older continuous potential \(\varphi\) with sufficiently small oscillation admits a unique equilibrium state. More precisely, if
\[
\sup \varphi - \inf \varphi < \frac{c_0}{2} - \kappa,
\]
where \(c_0\) is the non-uniform expansion constant and \(\kappa\) is the unstable Jacobian distortion bound, then \(\varphi\) has a unique equilibrium state. Hence Viana maps admit unique equilibrium states for all sufficiently small potentials.

\subsection{Benedicks--Carleson quadratic family}
In this section, we explain how \cref{thm:main-1d} applies to the rich class of quadratic maps known as the \emph{Benedicks--Carleson quadratic family}. As a consequence, we obtain existence and uniqueness of equilibrium states for every H\"older potential on the core interval of the maps in this family, as stated in \cref{BC}.

\subsubsection{\texorpdfstring{On Iterations of $1 - ax^2$ on $(- 1,1)$}{On Iterations of 1 - ax2 on (- 1,1)}}
The quadratic family
\[
f_a \colon [-1,1] \to [-1,1], \qquad f_a(x)=1-ax^2, \quad a\in[0,2],
\]
provides one of the most important examples in one-dimensional dynamics, and its statistical properties have been studied extensively. From the statistical point of view, the first major breakthrough was Jakobson's theorem \cite{Jak81}, which showed that for a positive Lebesgue measure set of parameters, the map exhibits stochastic behavior. More precisely, for such parameters there exists an absolutely continuous invariant measure (in this case, a physical measure) with positive Lyapunov exponent.

\begin{theorem}[Jakobson \cite{Jak81}]
    Consider the quadratic family
    \[
    f_a \colon [-1,1]\to[-1,1], \qquad f_a(x)=1-ax^2, \quad a\in[0,2].
    \]
    Then there exists a positive Lebesgue measure set \(\Delta\subset [0,2]\) such that for every \(a\in\Delta\), the map \(f_a\) admits an absolutely continuous invariant probability measure.
\end{theorem}

Each map \(f_a\) is an \(S\)-unimodal map: it is of class \(C^3\), has a unique critical point, and has negative Schwarzian derivative \cite{Dev21}. This family has received special attention because it already exhibits a remarkably rich range of dynamical phenomena, and at the same time serves as a model for more general non-uniformly expanding systems.

On the other hand, Lyubich \cite{Lyu97} and Graczyk--\'Swiatek \cite{GS97} later showed that the set of regular parameters, namely those for which the map has a hyperbolic periodic attractor, is open and dense. Thus the quadratic family exhibits at least two very different kinds of observable behavior: stochastic dynamics, characterized by expansion and chaotic statistical behavior, and regular dynamics, governed by a periodic attractor. These two regimes alternate through parameter space in a highly intricate way.

It was later shown that stochastic behavior could be concluded from enough expansion along the orbit of the critical value: the Collet--Eckmann condition, exponential growth of $|(f^n)'(1)|$, was enough to conclude a positive Lyapunov exponent of the system. A different approach to Jakobson’s Theorem in \cite{BC85} and \cite{BC91} focused specifically on this property: the set of Collet--Eckmann maps has positive measure. In what follows, the \emph{Benedicks--Carleson quadratic family} refers to the parameters sufficiently close to \(2\) considered in \cite{You92,Fre10}. This parameter class is not identified with the arbitrary positive-measure set \(\Delta\) in Jakobson's theorem above.

We emphasize that the growth of the derivative along the critical orbit is directly related to the sensitivity of nearby orbits to their initial conditions, which is one of the most characteristic features of chaotic dynamics.

\subsubsection{Main Properties}
We now verify that the Benedicks--Carleson quadratic family satisfies the hypotheses needed to apply our main theorem. We begin with topological exactness.

\begin{lemma}[Young {\cite[Lemma 2.1]{You92}}]
    For the Benedicks--Carleson parameters considered here and every nondegenerate interval \(I \subset [-1,1]\), there exists \(n=n(I)\) such that
    \[
    [f_a^2(0),f_a(0)] \subset f_a^n(I).
    \]
\end{lemma}

Recall that the interval $I_a: =[f_a^2(0),f_a(0)]=[1-a,1]$ is called the \emph{core interval}. In particular, this implies that \(f_a\) is topologically exact on its invariant core interval $I_a$.

We also need non-uniform expansion and slow recurrence to the critical set. These were established by Freitas.

\begin{theorem}[Freitas {\cite[Theorems 1.1 and 1.2]{Fre10}}]
    Assume that \(a\) is a Benedicks--Carleson parameter. Then:
    \begin{itemize}
        \item \(f_a\) is non-uniformly expanding for Lebesgue almost every point;
        \item \(f_a\) has slow recurrence to the critical set for Lebesgue almost every point.
    \end{itemize}
\end{theorem}

With these ingredients in hand, we are ready to prove \cref{BC}.

\subsubsection{Existence and Uniqueness of Equilibrium States}
For every Benedicks--Carleson parameter, the quadratic map \(f_a\) is topologically exact on its core interval \(I_a\), has a non-flat critical point at \(0\), and satisfies non-uniform expansion and slow recurrence to the critical point for Lebesgue almost every \(x\in I_a\). Thus the dynamical hypotheses of \cref{thm:main-1d} hold.

The map is smooth, all its periodic points are hyperbolic repelling, and the Collet--Eckmann condition gives \(\lvert(f_a^n)'(f_a(0))\rvert\to\infty\). Hence \cite[Theorem A]{LRL14b} implies that every H\"older continuous potential \(\varphi:I_a\to\mathbb R\) is hyperbolic:
\[
\max_{\nu\in\mathcal M(f_a|_{I_a})}\int\varphi\,d\nu
 <P_{\mathrm{top}}(\varphi;f_a|_{I_a}).
\]
Applying \cref{thm:main-1d}, we obtain a unique equilibrium state with positive entropy and the upper large deviations bound stated in \cref{ct}. This proves \cref{BC}.

\bibliographystyle{amsalpha}
\bibliography{Bibliography}

\end{document}